\documentclass[12pt]{article}
\usepackage{framed}
\usepackage{natbib}
\RequirePackage{kpfonts}
\RequirePackage[sf,bf,small,raggedright]{titlesec}
\usepackage{microtype}
\usepackage{textcomp}
\usepackage{amssymb}
\usepackage{bbm}
\usepackage{fancyhdr}
\usepackage{amsmath}
\usepackage{amsbsy,amsthm}
\usepackage{amscd}
\usepackage{latexsym}
\usepackage{graphicx}   
\usepackage{pdfsync}
\usepackage{blkarray}
\usepackage{multirow}
\usepackage{hyperref}
\usepackage{tcolorbox}

\newcommand{\R}{\mathbb{R}}

\newcommand{\inr}[1]{\left\langle #1 \right\rangle}

\newcommand{\N}{\mathbb{N}}

\newcommand{\E}{\mathbb{E}}
\newcommand{\M}{\mathbb{M}}

\newcommand{\cT}{\mathcal{T}}
\newcommand{\Borel}{\mathcal{B}}

\newcommand{\eps}{\varepsilon}

\newtheorem{theorem}{Theorem}[section]
\newtheorem{lemma}[theorem]{Lemma}
\newtheorem{corollary}[theorem]{Corollary}
\newtheorem{proposition}[theorem]{Proposition}
\theoremstyle{definition}
\newtheorem{remark}[theorem]{Remark}
\newtheorem{definition}[theorem]{Definition}
\newtheorem{example}[theorem]{Example}

\renewcommand{\epsilon}{\varepsilon}

\numberwithin{equation}{section}

\def \endproof
{{\mbox{}\nolinebreak\hfill\rule{2mm}{2mm}\par\medbreak}}

\newcommand{\ol}{\overline}

\newcommand{\wh}{\widehat}
\newcommand{\argmin}{\mathop{\mathrm{argmin}}}

\newcommand{\X}{\mathcal{X}}

\newcommand{\cE}{\mathcal{E}}

\newcommand{\EXP}{\mathbb{E}}
\newcommand{\PROB}{\mathbb{P}}

\newcommand{\cov}{\mathbb{C}{\rm ov}}

\newcommand{\lambdamax}{\lambda_{\text{max}}}

\usepackage[normalem]{ulem}  

\begin{document}

\title{Robust, sub-Gaussian mean estimators in metric spaces}
 
\author{Daniel Bartl \\[-0.1em]
\small
Department of Mathematics and \\ [-0.3em]
\small 
Department of Statistics \& Data Science,\\ [-0.3em]
\small
National University of Singapore 
\and
G\'abor Lugosi \\ [-0.1em]
\small
Department of Economics and Business, \\ [-0.3em]
\small
Pompeu  Fabra University 
\\ [-0.3em]
\small
ICREA, Pg. Lluís Companys 23,\\  [-0.3em]
\small
08010 Barcelona, Spain \\ [-0.3em]
\small
Barcelona School of Economics
\and
Roberto I. Oliveira\\ [-0.1em]
\small
Instituto de Matemática Pura e Aplicada \\[-0.3em]
\small
Estrada Dona Castorina, 110. \\ [-0.3em] \small Rio de Janeiro, RJ, Brazil \\ [-0.3em]
\and
Zoraida F. Rico \\ [-0.1em]
\small
Department of Decision Sciences, \\ [-0.3em]
\small
Bocconi University, Milan, Italy
 }

\maketitle

\begin{abstract}
Estimating the mean of a random vector from i.i.d.\ data has received considerable attention, and the optimal accuracy one may achieve with a given confidence is fairly well understood by now. 
When the data take values in more general metric spaces, an appropriate extension of the notion of the mean is the Fréchet mean. While asymptotic properties of the most natural Fréchet mean estimator (the empirical Fréchet mean) have been thoroughly researched, non-asymptotic performance bounds 
have only been studied recently. 
The aim of this paper is to study the performance of estimators of the Fréchet mean in general metric spaces under possibly heavy-tailed and contaminated data. 
In such cases, the empirical Fréchet mean is a poor estimator. 
We propose a general estimator based on high-dimensional extensions of trimmed means and prove general performance bounds.
Unlike all previously established bounds, ours generalize the optimal bounds known for Euclidean data.
The main message of the bounds is that, much like in the Euclidean case, the optimal accuracy is governed by two ``variance'' terms: a ``global variance'' term that is independent of the prescribed confidence, and a potentially much smaller, confidence-dependent ``local variance'' term. 
We apply our results for metric spaces with curvature bounded from below, such as Wasserstein spaces, and for uniformly convex Banach spaces.
\end{abstract}

\clearpage

\section{Introduction}

Estimating the expected value of a random vector from independent, identically distributed (i.i.d.) samples is one of the most fundamental problems in statistics. 
While the classical central limit theorem provides asymptotic guarantees for the convergence of the empirical mean, modern machine learning and high-dimensional statistics focus on \emph{non-asymptotic} theory: obtaining high-confidence estimates that hold for a given finite sample.
Despite the importance and apparent simplicity of the task, the problem poses significant (and perhaps unexpected) mathematical challenges.
In particular, sharp bounds for mean estimation in high-dimensional Euclidean spaces $(\mathbb{R}^m, \|\cdot\|_2)$ for a general sample that may be heavy-tailed or contaminated have only been established recently; see
\cite{AbZh22,DeLe22,DiKaKaLiMoSt19,Hop20,LeLe20,LuMe19,LuMe18a,LuMe16a,LuMe20,LuMe24,Min18,oliveira2023trimmedsamplemeansrobust,OlRi22}
for a sample of the recent literature, and \cite{lugosi2019mean} for a survey of related developments. 
Importantly, mean estimators achieving near-optimal performance are \emph{not} the sample mean (whose performance drastically decreases unless one imposes highly restrictive assumptions on the data), but significantly more complex functions of the data, based on  \emph{median-of-means tournaments} and multi-dimensional extensions of \emph{trimmed-mean estimators}.
Roughly speaking, the main message of these results is that, as long as the second moment of a random vector exists, there exist estimators of the mean that perform as well as if the distribution of the random vector was Gaussian.

With the rise of complex and structured data, there has been a significant shift toward understanding ``mean'' estimation in general metric spaces $(\mathcal{X}, d)$ rather than only in the Euclidean framework.
In such settings, the notion of a ``mean'' of a random vector $X$ with values in $(\X,d)$ must of course be generalized, with the most widely adopted approach being due to Fr\'echet~\cite{frechet}: if $\mathbb{E}[d^2(X, x)]$ is finite for some (and hence all) $x \in \mathcal{X}$, the \emph{Fr\'echet mean} $\mu$ is any minimizer of the map $x \mapsto \mathbb{E}[d^2(X, x)]$. 

In this paper, we work in settings where the Fr\'echet mean is unique, thus
\[\mu = \mathop{\rm argmin}_{x \in \mathcal{X}} \, \mathbb{E}[d^2(X,x)]~.\]
Note that the Fr\'echet mean is indeed a generalization of the Euclidean mean as clearly $\mu=\E[X]$ when $(\mathcal{X}, d) = (\mathbb{R}^m, \|\cdot\|_2)$.
The existence and uniqueness of the Fr\'echet mean  has been thoroughly studied; we refer the reader to Ziezold \cite{Ziezold1977} for early results on existence in quasi-metric spaces, Afsari \cite{Afs11} for a detailed analysis in Riemannian \( L^p \) settings---particularly under nonpositive curvature assumptions---and to Sturm \cite{Stu03} for results in metric spaces with curvature bounded above in the Alexandrov sense. See also Huckemann and Eltzner \cite{HuEl21} for recent developments in non-Euclidean settings.

A particularly important example of such a metric space $(\mathcal{X}, d)$ that arises frequently in data science (see, e.g., \cite{arjovsky2017wasserstein,carlier2010matching,genevay2018learning,solomon2015convolutional}) is the Wasserstein space $(\X,d)=(\mathcal{P}_2(\mathbb{R}^m), \mathcal{W}_2)$, the space of probability measures with finite second moment endowed with the Wasserstein distance.
This space has non-negative curvature and will often serve as a guiding example throughout this paper.
In the setting of the Wasserstein space, the Fréchet mean was first studied by Agueh and Carlier \cite{AgCa11} and is often called  Wasserstein barycenter.
It appears naturally in various contexts.
For instance, in image analysis, the Wasserstein barycenter is the natural mechanism to average images, see Peyré and Cuturi \cite{PeCu19} for a survey and further references.

The statistical estimation of the Fr\'echet mean based on $n$ i.i.d.\ observations $X_1, \ldots, X_n$ has been quite thoroughly studied, when the estimator is the \emph{empirical Fréchet mean}
 defined by
\begin{align}
\label{eq:into.def.empirical.mean}
    \ol\mu_n= \argmin_{x\in \X} \frac{1}{n} \sum_{i=1}^n d^2(X_i,x)~.
\end{align}
The empirical Fr\'echet mean is the obvious generalization of the ordinary empirical mean.
However, the analysis of this natural estimator  poses significantly more challenges than in the classical Euclidean setting. 
While several authors have studied the asymptotic properties of the empirical Fréchet mean (see 
\cite{AvMu24,BhPa03,EvJa24,hotz2024central,Jaf24,kroshnin2021statistical,santoro2023large,Sch22}),
finite-sample guarantees have only been obtained recently, see \cite{brunel2023concentration,brunel2025finite,YuPa23}.
For example, in the Wasserstein space $(\mathcal{P}_2(\mathbb{R}^m), \mathcal{W}_2)$, meaningful statistical guarantees were largely out of reach until recently. 
Earlier results (see Schötz \cite{Sch19}, Ahidar-Coutrix, Le Gouic, and Paris \cite{AhLePa20}), suffered from the curse of dimensionality---requiring a sample size on the order of $n \sim \exp(m)$ to obtain reliable estimates. (Such a requirement is clearly infeasible even in moderate-dimensional applications.)

A recent breakthrough by Le Gouic, Paris, Rigollet, and Stromme~\cite{LePaRiSt22} has fundamentally changed the perspective on Fr\'echet mean estimation in metric spaces by establishing statistical guarantees that overcome the curse of dimensionality. 
Their analysis considers a broad geometric setting in which $(\X, d)$ is an Alexandrov space (i.e., a complete geodesic space with curvature bounded from below), encompassing  all Riemannian manifolds and in particular the Wasserstein space. 
Under a synthetic geometric condition of the space $\X$ encoded in a constant $c(\X)$, they prove that the empirical Fr\'echet mean $\overline{\mu}_n$ satisfies
\begin{equation}
\label{eq:legouic}
 \mathbb{E}\left[ \, d^2(\overline{\mu}_n, \mu) \right]^{1/2} 
\leq c(\X) \cdot \sqrt{\frac{\mathbb{E} \, d^2(X, \mu)}{n}}~.
\end{equation}
Notably, when $(\X,d)$ is Wasserstein space, the constant $c(\X)$ can be controlled using tools from optimal transport theory, see Section \ref{sec:example.wasserstein} for more details.

These results established in \cite{LePaRiSt22} reveal that statistical estimation in non-linear and even high-dimensional spaces---such as Wasserstein space---is far more tractable than previously believed. 
While \eqref{eq:legouic} bounds the expected squared distance between the estimator $\ol{\mu}_n$ and the Fréchet mean $\mu$, in this paper we are primarily interested in bounds that hold with high probability. 
More precisely, for a given probability or error $\delta \in (0,1)$, our goal is to understand the best possible accuracy that one can achieve. 
In other words, our aim is to understand what the smallest radius $r_n$ is for which there exists a mean estimator $\wh\mu_n$ achieving 
\[
\mathbb{P}(d(\wh{\mu}_n, \mu) \leq r_n) \geq 1 - \delta~.
\]
In particular, this paper addresses the following  natural questions: 

\begin{tcolorbox}
  \textbf{Main questions:}
  \begin{itemize}
    \item Given \emph{sub-Gaussian} data $X_1, \ldots, X_n$ (for an appropriate notion of sub-Gaussianity) and a desired confidence level $1 - \delta$, what is the smallest radius $r_n$ such that the empirical Fréchet mean $\ol\mu_n$ satisfies
    \[
      \mathbb{P}(d(\overline{\mu}_n, \mu) \leq r_n) \geq 1 - \delta?
    \]
    \item Is it possible to construct an estimator $\widehat{\mu}_n$ for the Fr\'echet mean $\mu$ that achieves the same error as $\overline{\mu}_n$ in the sub-Gaussian setting, but under minimal assumptions on the data and even in the presence of outliers?
  \end{itemize}
\end{tcolorbox}

These questions are motivated by the above-mentioned recent progress in the Euclidean setting. 
The contribution of this work is a positive solution to both questions that significantly generalizes what is known for mean estimation in Euclidean spaces.

First, we identify an appropriate notion of sub-Gaussian random vectors $X$  in Alexandrov spaces and derive high-probability concentration bounds for the empirical Fr\'echet mean. 
In particular, we show that---much like in the Euclidean setting---, the smallest radius $r_n$ for which $\mathbb{P}(d(\overline{\mu}_n, \mu) \leq r_n) \geq 1 - \delta$, naturally decomposes into a sum of a \emph{global variance} component that is independent of the confidence parameter $\delta$, and a typically much smaller \emph{local variance} component.
Notably, when applying our general theory to the Euclidean setting $X\in\R^m$, our notion of sub-Gaussianity simplifies to the classical notion, and our estimate on $r_n$ coincides, up to a constant factor, with the standard one:
\begin{align}
    \label{eq:intro.r.n}
    r_n \sim \sqrt \frac{\EXP d^2(X,\mu)}{n} + \sqrt\frac{\lambda_{\max}(\Sigma) \log(1/\delta)}{n}~,
\end{align}
where $\Sigma$ is the covariance matrix of $X$, and $\lambdamax(\Sigma)$ denotes its largest eigenvalue. Note that since $\EXP d^2(X,\mu)=\rm{trace}(\Sigma)$, 
the first term is ``dimension-dependent'', and typically much larger than the second term, at least when $\delta$ is not exponentially small.
In the general setting of Alexandrov spaces, we obtain a similar expression for $r_n$, though the notion of the covariance matrix needs to be defined properly.
However---analogous to the Euclidean setting---the sub-Gaussian assumption is rarely satisfied in practice, and the performance of $\overline{\mu}_n$ deteriorates drastically for even mildly heavier-tailed data.

Second, and arguably more significantly, we design a so-called \emph{sub-Gaussian estimator} for the Fr\'echet mean.
That is, we construct an estimator $\widehat{\mu}_n$ (i.e., a measurable function from $\X^n$ to $\X$ that takes as input the sample and outputs an estimate for $\mu$) that satisfies
\[
  \mathbb{P}\left(d(\widehat{\mu}_n, \mu) \leq r_n \right) \geq 1 - \delta
\]
with the same choice of $r_n$  (up to a constant factor) as in the setting of sub-Gaussian data---but under minimal assumptions on the data.
In particular,  the data may come from a heavy-tailed distribution (having only a finite second moment), and some of the data points may even be adversarially corrupted.
Note that in this setting the statistical behavior of $\widehat{\mu}_n$ is in stark contrast to that of the empirical Fr\'echet mean $\overline{\mu}_n$, whose performance degrades rapidly when the data deviates (even slightly) from the sub-Gaussian regime.
Moreover, it is rather obvious that even a single (adversarially) corrupted data point $X_i$ may cause the empirical Fr\'echet mean to become arbitrarily inaccurate---an issue that led, for example, to median-based methods (see, e.g.\ \cite{carlier2024wasserstein,fletcher2009geometric}), even though in practice the mean is often the quantity of real interest.

Our results for Alexandrov spaces can be viewed as a genuine generalization of the classical Euclidean mean estimation problem, in that we establish robust, sub-Gaussian estimators that apply to Alexandrov spaces.

In addition to the results in the setting of Alexandrov spaces, we demonstrate that the methods naturally extend to the estimation of Fr\'echet means in uniformly convex Banach spaces of power type 2.
In particular, we design robust and sub-Gaussian estimators in that context as well.

\begin{remark}[\emph{A general median-of-means estimator.}]
\label{rem:MOM.general}
We conclude the introduction with a comparison to the work of Hsu and Sabato~\cite{HsSa16}, who, among other contributions, introduce a generalization of the median-of-means estimator to arbitrary metric-space-valued data.  
Adapting the results in \cite{HsSa16} for the purpose of illustration, suppose that the Fréchet mean $\mu$ exists and is unique, and define $R(\ell) := \mathbb{E}[ d(\overline{\mu}_\ell,\mu) ]$
to be the expected error of the empirical Fréchet mean (see~\eqref{eq:into.def.empirical.mean}) based on $\ell \geq 1$ samples.  
Then, for every $\delta \in (\exp(-\frac{n}{16}),\frac{1}{2})$, there exists an estimator $\widetilde{\mu}_n$ such that, with probability at least $1-\delta$,  
\begin{align}
    \label{eq:intro.MOM}
d(\widetilde{\mu}_n,\mu)
\leq 8 R\!\left( \frac{n}{20\log(1/\delta)} \right).
\end{align}
For completeness, we present the (very short) proof of this fact in Section~\ref{proof.MOM.general}.

Importantly, the bound in \eqref{eq:intro.MOM} is essentially never sharp in spaces with additional structure.
Even in the Euclidean case (or in the framework of~\cite{LePaRiSt22} where~\eqref{eq:legouic} holds), the best consequence of~\eqref{eq:intro.MOM} is that, with probability at least $1-\delta$, 
\[
d(\widetilde{\mu}_n,\mu)
\leq 36 c(\X) \sqrt{ \frac{ \mathbb{E}\, d^2(X,\mu) \, \cdot \,   \log(1/\delta)}{n} } .
\]
In particular, there is no error decomposition as in \eqref{eq:intro.r.n}, and the rate here is substantially worse than the optimal one (which our estimator achieves).

Thus, while the results of~\cite{HsSa16} apply to very general metric spaces,  the error bounds are suboptimal for spaces with some structure.
This is natural because the proof techniques are effectively one-dimensional and thus cannot yield optimal rates---even in Euclidean spaces. 
We refer to \cite{kim2025robust,YuPa23} for results of similar flavor.
\end{remark}

\subsection{Organization.}  
The paper is organized as follows. In Section \ref{sec:preliminary}, 
we introduce the contamination model and define trimmed means for functions.
We also
briefly discuss the necessity of regularity conditions to avoid problems of measurability and existence of the Fréchet mean. To this end, in the Appendix we introduce the notion of
`favorable spaces' that guarantees the existence of the Fréchet mean and the measurability of the empirical Fréchet mean. We also prove the existence of a measurable choice for the trimmed mean estimator (see Proposition \ref{prop:estimator.measurable}).

In Section \ref{sub:trimmedestimators}, 
we introduce a general estimator of the Fréchet mean, and establish some basic tools for their study.

In Section \ref{sec:geodesic}
we establish the main performance bound for the proposed estimator in general metric spaces with curvature bounded from below. Theorem \ref{thm:geodesic.main} provides high-probability bounds for the trimmed Fréchet mean under weak assumptions; more specifically, the favorable space condition and a uniform lower bound on the so-called ``hugging'' function. We also point out that that these general bounds are optimal in the case of separable Hilbert spaces. We also analyze the empirical Fréchet mean under appropriate sub-Gaussian conditions.

In Section \ref{sec:banach.space}, we consider another class of 
metric spaces. In particular, we analyze the proposed general Fréchet mean estimator in uniformly convex Banach spaces of power type 2, using the techniques developed in earlier sections. 

Section \ref{sec:examples} presents an example that illustrates the necessity of some of our assumptions. We also specialize our results to the setting of probability measures on $\mathbb{R}^m$ with the Wasserstein-2 distance, as well as Fréchet means for $\ell^p$-norms on $\mathbb{R}^m$, and derive general bounds in these cases.  

We conclude the paper by describing some directions for future research. 
\section{Preliminaries}
\label{sec:preliminary}

\subsection{Fréchet means}

Throughout this paper, we assume that $(\X,d)$ is a separable metric space and 
$X$ is a random variable taking values in the measurable space $(\X,\sigma(\X))$ (the $\sigma$-field $\sigma(\X)$ will be the Borel $\sigma$-field whenever $(\X,d)$ is a Polish metric space).

We assume that $\EXP d^2(X,x)<\infty$ for some $x\in \X$, and we define 
the Fréchet mean by 
\[
\mu = \mathop{\rm argmin}_{x \in \mathcal{X}} \, \mathbb{E}[d^2(X,x)]~.
\]
Throughout the paper, we assume that the minimum is well defined and unique.  For the existence of the minimum, it is sufficient that $(\X,d)$ is a \emph{favorable} metric space. This general notion is defined in Appendix \ref{proof.existence.barycenter}, where the existence of the minimum is also proven. Uniqueness of $\mu$ is a more special property that we verify in each of the special cases considered.

\subsection{Random samples and contamination}\label{sub:randomsamples}

Throughout this paper,  $X_1,\ldots,X_n$  denote independent and identically distributed (i.i.d.) random variables taking values in $\X$, and having the same distribution as $X$.
The task of the statistician is to estimate the Fréchet mean $\mu$ by a (measurable) function of the i.i.d.\ sample $X_1,\ldots,X_n$. 

We allow a certain fraction of the sample to be \emph{adversarially contaminated}, defined as follows.
For $\epsilon\in [0,1)$, an $\epsilon$-contamination of $(X_1,\ldots,X_n)$ is another set of $(\X,\sigma(\X))$-valued random variables $X_1^\epsilon,\ldots,X_n^\epsilon$, defined on the same probability space as the $X_i$, such that
\[\#\left\{i\in \{1,\ldots,n\} \,:\, X_i\neq X_i^\epsilon \right\}\leq \epsilon n~. \]
Note that we assume that the value of $\epsilon$ is known to the statistician.

\subsection{Definition and uniform concentration of the trimmed mean}\label{sub:definetrimmed}

Our estimator builds on high-dimensional extensions on the classical idea of the \emph{trimmed mean}, as developed by Lugosi and Mendelson \cite{LuMe20}, Oliveira, Orenstein, and Rico \cite{trimmedmean1d}, and Oliveira and Resende \cite{oliveira2023trimmedsamplemeansrobust}.
In particular, the techniques developed in \cite{oliveira2023trimmedsamplemeansrobust} serve as the basis of our general estimators and their analysis.
Here we recall the definition and its uniform concentration properties.

Given integers $n>0$ and $t\geq 0$ satisfying $2t<n$  we define  the trimmed mean of a function  $f\colon \X\to \R$ by
\[T_{n,t}(f)(x_1,\ldots,x_n) := \frac{1}{n-2t} \sum_{j=t+1}^{n-t} \left( (f(x_i))_{i=1}^n \right)_{(j)},
\quad x_1,\ldots,x_n\in\mathcal{X}.\]
Here, for a vector $a=(a_i)_{i=1}^n\in\R^n$, we denote by $(a_{(j)})_{j=1}^n$ its  order statistics (i.e., nondecreasing rearrangement); in particular $a_{(1)}= \min_{i=1,\ldots,n} a_i$ and  $a_{(n)}=\max_{i=1,\ldots,n} a_i$.

One can easily check that, if $f:\X\to\R$ is $(\sigma(\X)/\mathcal{B}(\R))$-measurable, then $T_{n,t}(f)$ is measurable with respect to the product $\sigma$-field over $\X^n$. 
Additionally, $T_{n,t}(\cdot)$ is monotone and homogeneous:  if $f\geq g$ pointwise, then $T_{n,t}(f)\geq T_{n,t}(g)$ pointwise, and if $\lambda\in\R$, then $T_{n,t}(\lambda f)=\lambda T_{n,t}(f)$.

Finally,  we set
\begin{align*}
\wh{T}_{n,t}(f)
&:=T_{n,t}(f)(X_1,\ldots,X_n) \quad\text{and} \\ 
\wh{T}^\epsilon_{n,t}(f)
&:=T_{n,t}(f)(X_1^\epsilon,\ldots,X_n^\epsilon). 
\end{align*}
Since $\wh{T}_{n,t}=\wh{T}_{n,t}^\epsilon$ for $\epsilon=0$, in what follows we shall state our statistical results only for $\wh{T}_{n,t}^\epsilon$.


The trimmed mean offers the significant advantage over the standard sample mean of achieving estimation accuracy for potentially heavy-tailed and corrupted data comparable to the accuracy of the standard sample mean with Gaussian data.
To formulate the result, let  $\mathcal{F}$  denote a class of measurable functions from $\mathcal{X}$  to  $\mathbb{R}$ and assume that $ \E|f(X)| < \infty$ for all  $ f \in \mathcal{F} $. 
The following result is a simplified version of Oliveira and Resende \cite[Theorem 3.1]{oliveira2023trimmedsamplemeansrobust}.

\begin{theorem}
\label{thm:trimmed.mean}
Let $\mathcal{F}$ be countable\footnote{This assumption can be relaxed whenever $\mathcal{F}$ is separable in a suitable sense; see e.g., the notion of $1$-compatibility in \cite{oliveira2023trimmedsamplemeansrobust}.}, let  $\delta\in (0,1)$,  set
\[t=t(n,\delta,\varepsilon):= \lfloor \epsilon n\rfloor +\lceil \ln(2/\delta)\rceil \vee \left\lceil \tfrac{\epsilon\wedge \left(\frac{1}{2} - \epsilon\right)}{2}\,n\right\rceil\]
and assume that $t<\frac{1}{2}n$.  
Define 
\begin{align*}
V_n&:= \EXP\sup_{f\in\mathcal{F}}\left|\frac{1}{n}\sum_{i=1}^nf(X_i) - \EXP\,f(X)\right|, \\
\nu_p&:=\sup_{f\in\mathcal{F}} \left(\EXP\,|f(X)-\EXP\,f(X)|^p \right)^{1/p}, \qquad (p\geq 1)~,
\end{align*}
and set
$C_\epsilon:=384 (1 + \frac{\epsilon}{\epsilon\wedge \left(\frac{1}{2}-\epsilon\right)})$.
Then, with probability at least $1-\delta$, 
\begin{align*}
\sup_{f\in\mathcal{F}} \left| \wh{T}^\epsilon_{n,t}(f) - \EXP\,f(X) \right| 
\leq & C_\epsilon \left(8V_n + \inf_{p\in[1,2]} \nu_p\,\left(\frac{\ln(3/\delta)}{n}\right)^{1- \frac{1}{p}} + \inf_{p\in[1,\infty)}\nu_p\,\epsilon^{1-\frac{1}{p}}\right)~.
\end{align*}
\end{theorem}

\section{A general estimator of the Fr\'{e}chet mean}
\label{sub:trimmedestimators}

In this section we define a general estimator of the Fréchet mean. The estimator is a version of the trimmed-mean estimators studied and analyzed by Oliveira and Resende \cite{oliveira2023trimmedsamplemeansrobust}. Here we establish measurability of the estimator in our context and recall a useful inequality from \cite{oliveira2023trimmedsamplemeansrobust} that will help us analyze the performance of the Fréchet-mean estimator.

Following the notation introduced in Section \ref{sub:randomsamples}, let $X_1,\ldots,X_n$ be i.i.d.\ copies of a random variable $X$ taking values in a favorable metric space $(\X,d)$ and let $X_1^\epsilon,\ldots,X_n^\epsilon$ be an adversarially contaminated version of $(X_1,\ldots,X_n)$. Let $0\leq t<n/2$ be an integer. The estimators studied in this paper are defined as solutions to the following minimax optimization problem:
\[\wh{\mu}^\epsilon_{n,t}\in \argmin_{b\in \X} \sup_{a\in\X} \wh{T}_{n,t}^\epsilon\left( d^2(b,\cdot) - d^2(a,\cdot)  \right)~,\]
where $\wh{T}_{n,t}^\eps$ is the trimmed mean operation defined in Section \ref{sub:definetrimmed}. 
To motivate the definition, notice that, for $t=0$ and $\epsilon=0$, this estimator is the usual empirical Fr\'{e}chet mean
\[\ol{\mu}_{n}\in \argmin_{b\in \X}\frac{1}{n}\sum_{i=1}^n\,d^2(X_i,b)~.\]

It is not immediately obvious that $\wh{\mu}^\epsilon_{n,t}$ is a measurable function of the data. In Appendix \ref{sub:measurabletrimmed} we show that measurability is not a problem in favorable spaces. In particular, we establish the following:

\begin{proposition}
\label{prop:estimator.measurable}
If $(\X,d)$ is favorable, then there exists a measurable mapping 
\[\wh{\mu}_{n,t} \colon (\X^n,\mathcal{B}(\X)^n)\to(\X,\mathcal{B}(\X))\]
satisfying the following:
for all $(x_1,\ldots,x_n)\in\X^n$,
\[\wh{\mu}_{n,t}(x_1,\ldots,x_n)
\in \argmin_{b\in \X} \sup_{a\in\X} {T}_{n,t}\left( d^2(b,\cdot) - d^2(a,\cdot)  \right)(x_1,\ldots,x_n)~.\]
\end{proposition}

\subsection{A useful lemma}

Next, we present a lemma based on \cite[Lemma 6.4]{oliveira2023trimmedsamplemeansrobust} that offers a convenient tool to bound the performance of the trimmed-mean estimator.

\begin{lemma}
\label{lem:localization}
  Let $\mu\in\X$ and $r>0$. Suppose that there exists $\gamma\geq 0$ such that the following hold:
 \begin{itemize}
 \item  for all $b\in \X$ with $d(b,\mu) > r$,
  \begin{equation}
\label{eq:lower1}
    \wh{T}_{n,t}^\epsilon\left(
      d^2(b,\cdot) - d^2(\mu,\cdot) \right)> \gamma~;
  \end{equation}
  \item
   for all $b\in \X$ with $d(b,\mu) \leq r$,
  \begin{equation}
\label{eq:lower2}
    \wh{T}_{n,t}^\epsilon\left(  d^2(b,\cdot) - d^2(\mu,\cdot) \right) \ge - \gamma~.
  \end{equation}
 \end{itemize}
	Then $d(\wh{\mu}^\epsilon_{n,t},\mu) \le r$.
\end{lemma} 
\begin{proof} Write $\wh{\mu}:=\wh{\mu}^\epsilon_{n,t}$ for simplicity.
  Define, for all $b\in \X$,
  \[
\psi(b) := \sup_{a\in \X} \wh{T}_{n,t}^\epsilon\left(
     d^2(b,\cdot) - d^2(a,\cdot) \right)
 \]
 so that $\wh{\mu}$ is a minimizer of $\psi$.
 We show that $\psi(b)>\psi(\mu)$ for all $b\in\mathcal{X}$ satisfying $d(b,\mu)>r$, from which the claim follows.
 
 Indeed, let  $b\in \X$ satisfy that $d(b,\mu)> r$.
 Then, by \eqref{eq:lower1}, 
\[
   \psi(b) \ge \wh{T}_{n,t}^\epsilon\left(d^2(b,\cdot) - d^2(\mu,\cdot) \right) > \gamma~.
\]
	We now claim that $\psi(\mu)\le \gamma$. 
	To that end, let $a\in \X$.
	If $d(a,\mu)\leq r$, then  $\wh{T}_{n,t}^\epsilon (d^2(\mu,\cdot) - d^2(a,\cdot) )\leq \gamma$ by \eqref{eq:lower2} and the fact that the operator $\widehat{T}_{n,t}^\eps(\cdot)$ is homogeneous. If, on the other hand, $d(a,\mu)>r$, then   $\wh{T}_{n,t}^\epsilon(d^2(\mu,\cdot) - d^2(a,\cdot) )<-\gamma$ by \eqref{eq:lower1} and homogeneity again.
	Hence $\psi(\mu)\leq \gamma$ by the definition of $\psi$ and homogeneity of the trimmed mean.
\end{proof}

\section{Spaces with curvature bounded from below}
\label{sec:geodesic}

In this section, we consider the case of Fr\'{e}chet mean estimation for metric spaces whose curvature is bounded from below. The precise statement of our results requires several definitions, that we recall in Section \ref{sub:curvature.bounded.below.defs} below along with some useful facts. In Section \ref{sub:frechetcurvature}, we present our main results for Fr\'{e}chet mean estimation and compare them to those of Le Gouic, Paris,  Rigollet, and Stromme \cite{LePaRiSt22}.
Proofs are deferred to the last subsections.  

\subsection{Metric spaces with curvature bounded from below: facts and definitions}\label{sub:curvature.bounded.below.defs} 
 We mostly follow the exposition of Le Gouic, Paris,  Rigollet, and Stromme \cite{LePaRiSt22} and also refer to Alexander, Kapovitch, and Petrunin \cite{alexander2023} and Chewi, Niles-Weed, and Rigollet \cite{ChewiNilesWeedRigollet2025} for additional material and background. 

\subsubsection{Comparisons and curvature}
Curvature upper and lower bounds in metric spaces are defined by comparison with the usual model spaces that we now recall. Given $\kappa\in\R$, we let $(M^2_\kappa,d_\kappa)$ denote the two-dimensional manifold with sectional curvature equal to $\kappa$ everywhere, with its intrinsic Riemannian metric $d_\kappa$. Concretely,
\begin{itemize}
\item for $\kappa>0$, $M^2_\kappa$ is a sphere of radius $1/\sqrt{\kappa}$, and $d_\kappa$ is its standard Riemannian metric;
\item for $\kappa<0$, $M^2_\kappa$ is the upper half-plane with the Riemannian metric $(dx^2+dy^2)/y^2|\kappa|$;
\item for $\kappa=0$, $M^2_\kappa=\R^2$ with the standard Euclidean metric.
\end{itemize}
We define $D_\kappa=+\infty$ for $\kappa\leq 0$ and $D_\kappa:=\pi/\sqrt{\kappa}$ for $\kappa>0$. 

Now consider a metric space $(\X,d)$. We will assume that this space is {\em a complete geodesic space}, meaning that, for any distinct $x,y\in\X$, there exists a curve $\gamma:[a,b]\to\X$  with $a<b$, $\gamma(a)=x$, $\gamma(b)=y$, and
\[\forall a\leq u<v\leq b\,:\,d(\gamma(u),\gamma(v)) = \left(\frac{v-u}{b-a}\right)\,d(x,y)~.\]
We call such a curve a \emph{constant-speed geodesic} in $\X$ from $x$ to $y$. 

A triangle in $\X$ is a set of three points $\{x,y,z\}\subset \X$. The perimeter of this triangle is simply
\[{\rm per}\{x,y,z\}:=d(x,y)+d(y,z)+d(x,z)~.\]
It can be shown that, given any $\kappa\in\R$ and any triangle with ${\rm per}\{x,y,z\}<2D_\kappa$, one can find a {\em comparison triangle} $\{\bar{x},\bar{y},\bar{z}\}\subset M_\kappa^2$ with the same interpoint distances as $x,y,z$. We say that $(\X,d)$ has {\em curvature lower bounded by $\kappa$}, and write ${\rm curv}(\X)\geq \kappa$, if for all triangles $\{x,y,z\}\subset \X$ with ${\rm per}(x,y,z)<2D_\kappa$, all comparison triangles $\{\bar{x},\bar{y},\bar{z}\}\subset M_\kappa^2$, and all constant speed geodesics $\gamma:[0,1]\to \X$ from $x$ to $y$, we have
\[\forall 0\leq u\leq 1\,:\,  d(z,\gamma(u))\geq d_\kappa(\bar{z},\bar{\gamma}(u))~,\]
where $\bar{\gamma}:[0,1]\to M_\kappa^2$ is the unique constant speed geodesic in $M_\kappa^2$ from $\bar{x}$ to $\bar{y}$. If the inequality in the preceding display is reversed, we say that $(\X,d)$ has {\em curvature upper bounded by $\kappa$}, and write ${\rm curv}(\X)\leq \kappa$.

\subsubsection{Tangent cones and hugging functions}

Metric spaces $(\X,d)$ with ${\rm curv}(\X)\geq \kappa$ for some $\kappa\in\R$ have a rich local structure. Given $\mu\in\X$, one can find a {\em tangent cone} $T_\mu\X$, which is a complete metric space with a certain metric 
\[(u,v)\in (T_\mu\X)^2\mapsto \|u-v\|_\mu\]
and a distinguished zero element $o_\mu\in T_\mu\X$. We also define an ``inner product''~operation on $T_\mu\X$ as follows: 
\begin{equation}\label{eq:tangentip}(u,v)\in (T_\mu\X)^2\mapsto \langle u,v\rangle_\mu := \frac{1}{2}\left(\|u-o_\mu\|_\mu^2 + \|v-o_\mu\|_\mu^2 - 2\|u-v\|^2_\mu\right)~.\end{equation}
Roughly, elements of $T_\mu\X$ correspond to geodesics emanating from $\mu$ at different speeds. This is similar to standard tangent spaces in Riemannian geometry, and the notation introduced above draws attention to the similarity between the two concepts. While in general the tangent cone is not a vector space, or even a geodesic space, there is a rich-enough analogy with the Riemannian setting that some concepts relating tangent spaces to manifold geometry can be adapted to the setting of metric spaces with curvature bounded below, see \cite[Part III]{alexander2023}. The proofs of the main results of \cite{LePaRiSt22} and of our own Theorem \ref{thm:geodesic.main} are examples of this approach. 

For $(u,v)\in (T_\mu\X)^2$ as above, the inner product $\langle u,v\rangle_\mu$ has an alternative expression in terms of the so-called Alexandrov angle~$\sphericalangle_\mu^\kappa(\gamma,\sigma)\in [0,\pi]$ between unit-speed geodesics $\gamma,\sigma$ corresponding to $u$ and $v$:
\[\langle u,v\rangle_\mu = \|u-o_\mu\|_\mu\,\|v-o_\mu\|_\mu\,\cos\sphericalangle_\mu^\kappa(\gamma,\sigma).\]
We do not define $\sphericalangle_\mu^\kappa(\gamma,\sigma)$ here, but note that the above formula implies the ``Cauchy-Schwarz inequality'':
\begin{equation}\label{eq:CSinthecone}\forall (u,v)\in (T_\mu\X)^2\,:\, |\langle u,v\rangle_\mu| \leq \|u-o_\mu\|_\mu\,\|v-o_\mu\|_\mu~.\end{equation}

We are especially interested in concepts related to first-order differential calculus. This requires the introduction of a {\em log map} $\log_\mu:\X\to T_\mu\X$ satisfying the following property:
\begin{equation}\label{eq:logpdistance}\forall x\in\X\,:\, \|\log_\mu x - o_\mu\|_\mu = d(x,\mu)~.\end{equation}
In particular, it follows that $o_\mu=\log_\mu \mu$. 
 The $\log_\mu$ map can be chosen to be measurable with respect to the $\sigma$-field generated by balls in the tangent cone, and this suffices for the operations described in the remainder of the paper to be measurable. Notice that $\log_\mu x = x-\mu$ when $(\X,d)$ is a Hilbert space.

One way to relate the $\log_\mu$ map to properties of $\X$ is through the {\em hugging function} \cite[equation (3.1)]{LePaRiSt22}. Given $x,y,\mu\in\X$ with $\mu\neq y$, we define
\begin{equation}\label{eq:huggingoriginal}
k_{\mu}^y(x):=1 - \frac{\|\log_\mu x - \log_\mu y\|_\mu^2-d^2(x,y)}{d^2(\mu,y)}~.\end{equation}

The main finding of Le Gouic, Paris, Rigollet, and Stromme \cite{LePaRiSt22} is that assuming a lower bound on the hugging function is both sufficient for proving fast convergence of empirical Fr\'{e}chet means and a natural condition in examples. We close this subsection by considering two families of such examples. The first one includes all Hilbert spaces and also all Riemannian manifolds whose sectional curvature is between $\kappa\leq 0$ and $0$.

\begin{theorem}[See equation (2.4) in \cite{LePaRiSt22}]\label{prop:huggingnonpositive} If $\kappa\leq {\rm curv}(\X)\leq 0$ for some $\kappa\leq 0$, then $d(x,y)\geq \|\log_\mu x - \log_\mu y\|_\mu$ for all $x,y,\mu\in\X$. As a result, $k_\mu^y(x)\geq 1$ always holds.\end{theorem}

The next theorem is useful when dealing with Wasserstein spaces and manifolds of nonnegative sectional curvature. 
It relies on the concept of extendible geodesics  (see \cite{AhLePa20,LePaRiSt22}), introduced below.

\begin{definition}Given $\lambda_{\rm in},\lambda_{\rm out}\geq 0$, a constant-speed geodesic $\gamma:[0,1]\to \X$ is said to be $(\lambda_{\rm in},\lambda_{\rm out})$-extendible if there exists a constant speed geodesic ${\gamma}^+:[-\lambda_{\rm in},1+\lambda_{\rm out}]\to \X$ with $\gamma^+\mid_{[0,1]}=\gamma$.
\end{definition}

\begin{theorem}[Theorem 4.2 in \cite{LePaRiSt22}]\label{thm:hugging.nonnegative} Assume that ${\rm curv}(\X)\geq 0$ and let $\mu,x\in\X$. Given $\lambda_{\rm in},\lambda_{\rm out}>0$, assume that there exists a $(\lambda_{\rm in},\lambda_{\rm out})$-extendible constant-speed geodesic from $\mu$ to $x$. Then
\[\inf_{y\in\X} k_\mu^y(x)\geq \frac{\lambda_{\rm out}}{1+\lambda_{\rm out}} - \frac{1}{\lambda_{\rm in}}~.\]
\end{theorem}

Following \cite{LePaRiSt22}, it is instructive to visualize extendible geodesics for subsets $ \X\subset \mathbb{S}^{d-1}$ of the unit sphere.
Let $\mu,x\in \mathbb{S}^{d-1}$ be any two points. 
Then the geodesic between $\mu$ and $x$ is $(\lambda_{\rm in},\lambda_{\rm out})$-extendible, if $d(\mu,x)(1+\lambda_{\rm in} + \lambda_{\rm out})\leq \pi$; thus (by Theorem \ref{thm:hugging.nonnegative}) it holds that $\inf_{y\in\X} k_{\mu}^y(x)>0$ whenever $1+\lambda_{\rm in}+\lambda_{\rm out}\geq 4$.
The latter means that $\X$ needs to be contained in a spherical cap around $\mu$ of radius $\pi/4$. 

In Section \ref{sec:example.wasserstein}, we further detail how to exploit Theorem \ref{thm:hugging.nonnegative} in the context of Wasserstein spaces.

\subsection{Performance of trimmed Fr\'{e}chet mean estimators}\label{sub:frechetcurvature}

In this section, we establish the main performance bound for the trimmed-mean estimator of the Fréchet mean defined in Section \ref{sub:trimmedestimators}. We 
make the following assumptions:

\begin{itemize}
\item $(\X,d)$ is a favorable metric space in the sense of Definition \ref{def:favorable};
\item $X_1,\ldots,X_n$ are independent copies of a $\X$-valued random variable $X$ such that $\E\,d^2(X,x_0)<+\infty$ for some $x_0\in\X$, and $X_1^\epsilon,\ldots,X_n^\epsilon$ is an $\epsilon$-contamination of the sample $X_1,\ldots,X_n$;
\item $(\X,d)$ has curvature bounded from below in the sense of Section \ref{sub:curvature.bounded.below.defs};
\item letting $\mu$ denote a Fréchet mean of $X$, we assume a uniform lower bound on the hugging function of $(\X,d)$ at $\mu$:\[\inf_{x,y\in\X}k_\mu^y(x)\geq k_{\min}>0~.\]\end{itemize}

One of the main results of Le Gouic, Paris, Rigollet, and Stromme \cite[Theorem 3.3]{LePaRiSt22} is that, under the above conditions, $\mu$ is the unique Fr\'{e}chet mean of $X$ and the empirical Fréchet mean
\[\overline{\mu}_n\in \mathop{\rm argmin}_{b\in\X}\frac{1}{n}\sum_{i=1}^nd^2(X_i,b)\]
satisfies the bound
\begin{equation}\label{eq:expectationboundLePaRiSt}\E\,\big [d^2(\overline{\mu}_n,\mu)\big]\leq \frac{4\,\E\,d^2(X,\mu)}{k_{\min}^2n}~.\end{equation}

Our main general result establishes an upper bound for 
the performance of the trimmed Fr\'{e}chet mean estimators that holds in high probability, under weak assumptions.

To formulate the result, set 
\[ \nu_p:= \sup_{x\in\X \setminus\{\mu\}}  \, \frac{1}{d(\mu,x)} \left( \E |\langle\log_\mu(X),\log_\mu(x)\rangle_\mu|^p \right)^{1/p}  \]
for $p\geq 1$ and define $\sigma_{\rm w}:=\nu_2$ to be the \emph{local variance}.

\begin{theorem}
\label{thm:geodesic.main}
	Let $\delta\in(0,1)$ and assume that
	\[t=t(n,\epsilon,\delta):=\lfloor \epsilon n\rfloor +\lceil \ln(2/\delta)\rceil \vee \left\lceil \tfrac{\epsilon\wedge \left(\frac{1}{2} - \epsilon\right)}{2}\,n\right\rceil < \frac{n}{2}~.\]
	Then the estimator $\wh{\mu}=\wh{\mu}_{n,t}^\epsilon$ defined in Section \ref{sub:trimmedestimators} satisfies that, with probability at least $1-\delta$,
	\begin{align}
	\label{eq:est.d.main}
	 d\left(\wh{\mu},\mu\right)
	\leq \frac{C_\varepsilon}{ k_{\rm min}}  \left(  8\sqrt \frac{ \E[d^2(X,\mu)]}{n} + \sqrt \frac{ \sigma^2_{\rm w} \log(3/\delta) }{n}  + \inf_{p\in[1,\infty)}  \nu_p \epsilon^{1-1/p}\right)~,
	\end{align}
	where $C_\epsilon:=928 (1 + \frac{\epsilon}{\epsilon\wedge \left(\frac{1}{2}-\epsilon\right)})$.
\end{theorem}

The proof of the theorem may be found in Section \ref{proof:geodesic.main}.
Several comments are in order.

\begin{remark}[\emph{Breakdown point}]
\label{rem:breakdown.point.geodesic}
Theorem \ref{thm:geodesic.main} in particular establishes that the so-called breakdown point of our estimator occurs at a corruption level of  $\varepsilon \to \frac{1}{2}$.
Furthermore, for smaller values of $\varepsilon$---for example, when $\varepsilon \leq \frac{1}{4}$---the required conditions simplify to $\lceil \ln(2/\delta) \rceil \leq \frac{n}{4}$, which is satisfied whenever $\delta \geq 2\exp(-n/5)$ and $n \geq 12$. 
Moreover, we also have that $C_\varepsilon \leq 1160$ for $\varepsilon\leq \frac{1}{4}$.
\end{remark}

\begin{remark}[\emph{Weaker condition on hugging function}]
	In Theorem \ref{thm:geodesic.main} it suffices to assume that $\inf_{y\in\X} k_\mu^y(X_i^\epsilon)\geq k_{\min}$ for all $i=1,\ldots,n$.
	In particular, if the sample is uncontaminated (i.e., $\epsilon=0$), it suffices to assume that $ \inf_{y\in\X} k_\mu^y(X)\geq k_{\min}$ almost surely.
\end{remark}

\begin{remark}[\emph{Statistical optimality of Theorem \ref{thm:geodesic.main}}]
	We first note that since $\sigma_{\rm w}^2= \nu_2^2$ by definition, clearly
	\[ \inf_{p\in[1,\infty)} \epsilon^{1-1/p} \nu_p
	\leq \sqrt{ \epsilon }  \cdot \sigma_{\rm w}\]
	and thus the estimate \eqref{eq:est.d.main} in Theorem \ref{thm:geodesic.main} implies that, with probability at least $1-\delta$,
	\begin{align}
	\label{eq:bary.v.simplified}
	d(\wh{\mu},\mu)
	\leq \frac{C_\epsilon}{ k_{\rm min}}  \left(  8  \sqrt \frac{ \E[d^2(X,\mu)]}{n} + \sqrt \frac{ \sigma^2_{\rm w} \log(3/\delta) }{n}   +   \sqrt{ \epsilon }  \cdot \sigma_{\rm w} \right)~.
	\end{align}
	In the case when $\X$ is a separable Hilbert space and $d$ is the Hilbert-norm, this estimate coincides with the known and optimal bounds for mean-estimation (from corrupted data) in Hilbert spaces, see, e.g., \cite{LuMe20} and \eqref{eq:trace.cov} below.
	In particular, the dependence on  the corruption level $\varepsilon$ in \eqref{eq:bary.v.simplified} is known to be optimal in that setting.
	 At the same time, once the distribution of $X$ has lighter tails, the  estimate on $\inf_{p\in[1,\infty)} \epsilon^{1-1/p} \nu_p$ and thus also on \eqref{eq:bary.v.simplified} can be refined.
	For instance, if $X$ is $L$-sub-Gaussian (in the sense described in Definition \ref{def:subgaussian.geodesic} below), then a standard computation shows that 
	 \begin{align*}
	 \inf_{p\in[1,\infty)} \epsilon^{1-1/p} \nu_p
	 \leq C L  \varepsilon\sqrt{\log\left(\frac{3}{\varepsilon}\right)} \cdot \sigma_{\rm w}~,
	 \end{align*}
	 where $C$ is a universal constant.
	 Thus one may improve  $\sqrt{\varepsilon}$ in \eqref{eq:bary.v.simplified} to $CL \varepsilon\sqrt{\log(3/\varepsilon)}$.
	 Again, the resulting estimate is optimal in the class of $L$-sub-Gaussian distributions.
\end{remark}

\begin{remark}[\emph{On the two variances appearing in Theorem \ref{thm:geodesic.main}}]
It is important to highlight the difference between the two variances appearing in the estimate in Theorem \ref{thm:geodesic.main}: the ``global variance'' $\E d^2(X,\mu)$ and the---typically much smaller--- ``local variance'' $\sigma_{\rm w}^2$.
Indeed, for clarity, consider the case where $T_\mu\X$ is a proper Hilbert space (e.g., when $\X$ is a Riemannian manifold). Then $\log_\mu(X)$ is a random vector in the Hilbert space $T_\mu\X$, and we can define its covariance matrix 
\[\Sigma := \cov[\log_\mu(X)]~.\]
By definition, $\sigma_{\rm w}^2 = \lambda_{\max}(\Sigma)$, while a direct computation using~\eqref{eq:logpdistance} yields $\E d^2(X,\mu) = \operatorname{trace}(\Sigma)$.  
Thus, $\sigma_{\rm w}^2$ is \emph{dimension-independent} and it is always bounded above by $\E d^2(X,\mu)$, which can be as large as $m \cdot \sigma_{\rm w}^2$, where $m$ is the dimension of $T_\mu\X$.
In this setting, Theorem \ref{thm:geodesic.main} implies that, with probability at least $1-\delta$,
\begin{align}
\label{eq:trace.cov}
d(\wh{\mu},\mu)
	\leq \frac{C_\epsilon}{ k_{\rm min}}  \left(  8  \sqrt \frac{ {\rm trace}(\Sigma)}{n} + \sqrt \frac{ \lambda_{\max}(\Sigma)\log(3/\delta) }{n}   +   \sqrt{ \epsilon }  \cdot \sqrt{\lambda_{\max}(\Sigma)} \right)~.
\end{align}
\end{remark}



\subsection{Behaviour of the empirical Fréchet mean under sub-Gaussian tails}
While the primary focus of this paper is on developing estimators that are robust to corrupted and heavy-tailed data, our techniques also yield improved estimates for the standard empirical Fréchet mean 
\[ \overline{\mu}_n \in \argmin_{b\in\X}  \frac{1}{n}\sum_{i=1}^n d^2( X_i,b) \]
in the case of light-tailed data.

\begin{definition}
\label{def:subgaussian.geodesic}
A random vector $X$ with Fréchet mean $\mu$ is  $L$-\emph{sub-Gaussian} if for all $b\in\X$ and all $u\geq 0$,
	\[ \mathbb{P}\left( |\langle \log_{\mu}(X) , \log_\mu(b)  \rangle_{\mu}| \geq u  \| \langle \log_{\mu}(X) , \log_\mu(b)  \rangle_{\mu}\|_{L_2} \right)
	\leq 2\exp\left( - \frac{u^2}{2 L^2} \right)~.\]
\end{definition}

Let us emphasize that this definition coincides with the standard definition when $(\X,d)$ is a Hilbert space.
In particular, in that setting, any Gaussian random vector is $L$-sub-Gaussian with constant $L=1$.

We recall that $(X_i)_{i=1}^n$ is an i.i.d.\ sample of $X$.
Moreover, $(\X,d)$ is favorable and $\E \, d^2(X,x_0)<\infty$ for some $x_0\in\X$.

\begin{theorem}
\label{thm:geodesic.subgaussian}
	Let $\delta\in(0,1)$ satisfy that $\delta\geq \exp(-n)$ and suppose that $\inf_{y\in\X}k_\mu^y(X)\geq k_{\min}$ almost surely.
	Then $\mu$ is unique, and with probability at least $1-\delta$,
\[ d\left( \overline{\mu}_n ,\mu \right)
\leq \frac{L C}{k_{\min}}  \left( \sqrt  \frac{ \E \, d^2(X,\mu) }{  n} +  \sqrt{ \frac{ \sigma^2_{\rm w}  \log(3/\delta)}{ n}}  \right)~,  \]
	where $C$ is an absolute constant.
\end{theorem}

The theorem is proved in Section \ref{proof:geodesic.subgaussian}.
	
	It is important to underscore the distinction between Theorem \ref{thm:geodesic.subgaussian} and the corresponding result in \cite{LePaRiSt22}, see Theorem 3.5 therein.
	In \cite{LePaRiSt22} a random vector $X$ with Fréchet mean $\mu$ is called  sub-Gaussian with ``variance proxy'' $\zeta>0$ if 
\begin{align}
\label{eq:subgauss.JEMS}
\E \exp\left(\frac{d^2(\mu,X)}{2\zeta^2}\right) \leq 2~. 
\end{align}
It is shown that for a random vector $X$ that satisfies  \eqref{eq:subgauss.JEMS}, with probability $1-\delta-\exp(-c_1n)$,
\begin{align}
\label{eq:tail.bound.JEMS}
    d(\overline{\mu}_n,\mu)
\leq c_2\zeta  \sqrt{\frac{\log(3/\delta)}{n}}~ .
\end{align}
Here $c_1$ and  $c_2$ are constants that depend on the hugging function and the lower bound on the curvature of the space.\footnote{
At the end of the proof of Theorem 3.5 in \cite{LePaRiSt22} there is a typo and their result states that  \eqref{eq:tail.bound.JEMS} is true when $c_2\zeta $ is replaced by $\frac{1}{c_2\zeta}$.}
Crucially, the definition given in \eqref{eq:subgauss.JEMS} is  \emph{dimension-dependent} and the resulting estimate in \eqref{eq:tail.bound.JEMS} may be significantly worse than that obtained in Theorem \ref{thm:geodesic.subgaussian}.
As a concrete example, consider $(\X,d)=(\R^m,\|\cdot\|_2)$ and let $X$ be a standard Gaussian vector.
Then $X$ is $L$-sub-Gaussian with constant $L=1$ in the sense of our definition.
On the other hand, by Jensen's inequality, clearly  $\zeta^2\geq \frac{1}{\ln(4)} \E d^2(X,\mu) = \frac{1}{\ln(4)}m $.
In particular, \eqref{eq:tail.bound.JEMS} implies that with probability at least $1-\delta$,
\[ d(\overline{\mu}_n,\mu) \leq C \sqrt{ \frac{m\log(3/\delta)}{n} }~, \]
whereas Theorem \ref{thm:geodesic.subgaussian} shows that with probability at least $1-\delta$,
\[ d(\overline{\mu}_n,\mu) \leq C\left(  \sqrt{ \frac{m}{n} }  + \sqrt{ \frac{\log(3/\delta)}{n} }  \right)~, \]
which is a substantial improvement for small  values of $\delta$.

Let us also mention that in the recent work of Brunel and Serres \cite{brunel2025finite}, high-probability bounds for the empirical barycenter of i.i.d.\ data are established under certain sub-Gaussian assumptions, in two settings.
First, when dealing with general geodesic spaces, they introduce a sub-Gaussian condition that is different from ours. 
In particular, it appears  to be stronger and dimension-dependent.
For example, they show that if  $(\mathcal{X}, d) = (\mathbb{R}^m, \|\cdot\|_2)$  and $X$ is a vector with independent $L$-sub-Gaussian coordinates (in the classical,  one-dimensional sense), then $X$  is  $mL$-sub-Gaussian in their sense; see \cite[Remark 1]{brunel2025finite}.
In contrast, under our definition, the same vector can be readily verified to be $L$-sub-Gaussian.
Next, when specialized to Riemannian manifolds (with curvature bounded from above), the sub-Gaussian condition considered in \cite{brunel2025finite} aligns with ours. However, their resulting estimation bounds are strictly weaker---either due to a non-optimal prefactor in front of the $\E d^2(X, \mu)$  term, or due to a dimension-dependent alternative to $\sigma_{\rm w}^2$ (essentially, if  $(\mathcal{X}, d) = (\mathbb{R}^m, \|\cdot\|_2)$ and $X$ is the standard Gaussian, their local variance scales linearly with $m$ while ours is constant); see \cite[Theorems 11, 12]{brunel2025finite}.
    
\subsection{Proof of Theorem \ref{thm:geodesic.main}}\label{proof:geodesic.main}

To prove Theorem \ref{thm:geodesic.main}, we first collect some preliminary results. The first one is a strong-convexity-type estimate based on the hugging function. 

\begin{proposition}
The following estimate holds for all $x,y\in\X$:
\begin{equation}\label{eq:huggingalmostconv}d^2(x,y)-d^2(x,\mu) \geq k_{\min}\,d^2(\mu,y)- 2 \langle\log_\mu x ,\log_\mu y\rangle_\mu~.
\end{equation}
\end{proposition}
\begin{proof}
It suffices to consider $y\neq \mu$. Recalling the definition of $\langle\cdot,\cdot\rangle_\mu$ in \eqref{eq:tangentip} and the identity \eqref{eq:logpdistance}, we obtain
\[\|\log_\mu x - \log_\mu y\|_\mu^2 = d^2(x,\mu) + d^2(\mu,y) - 2 \langle\log_\mu x ,\log_\mu y\rangle_\mu~.\]
Combining this with the definition of the hugging function \eqref{eq:huggingoriginal}, we obtain that, for $x,y\in\X$ with $y\neq \mu$, 
\begin{equation*}k_{\mu}^y(x) = \frac{d^2(x,y)-d^2(x,\mu) + 2 \langle\log_\mu x ,\log_\mu y\rangle_\mu}{d^2(\mu,y)}~,
\end{equation*}
or equivalently,
\begin{equation*}d^2(x,y)-d^2(x,\mu)  = k_{\mu}^y(x)\,d^2(\mu,y)- 2 \langle\log_\mu x ,\log_\mu y\rangle_\mu~.\end{equation*}
By assumption, $k_{\mu}^y(x)\geq k_{\min}$, and this clearly implies \eqref{eq:huggingalmostconv} for $y\neq \mu$.\end{proof}

The next result compiles several results from \cite{LePaRiSt22} on the ``first order calculus''~of the function $\E\,d^2(X,\cdot)$ around its minimizer. 

\begin{proposition}
\label{prop:stolenfromLePaRiSt} Let $\mathbb{P}$ denote the probability distribution of $X$. Then there exists a subset $L_\mu\X\subset T_\mu\X$ containing $\log_\mu({\rm supp(\mathbb{P})})$\footnote{Here ${\rm supp}(\mathbb{P})$  denotes the support of $\mathbb{P}$, i.e., the smallest closed set of full measure. } which is a Hilbert space with zero element $o_\mu$, inner product given by the restriction of $\langle\cdot,\cdot\rangle_\mu$, and norm given by $\|h\|_\mu = \|h-o_\mu\|_\mu$. For any probability measure $\mathbb{Q}$ with ${\rm supp}(\mathbb{Q})\subset {\rm supp}(\mathbb{P})$, 
\begin{equation}\label{eq:Pettisuse}\forall b\in\X\,:\,\int_{\X}\langle \log_\mu x,\log_\mu b\rangle_\mu\,\mathbb{Q}(dx)=\left\langle\int_{\X}\log_\mu x\,\mathbb{Q}(dx),\log_\mu b\right\rangle_\mu,\end{equation}
where $\int_{\X}\log_\mu x\,\mathbb{Q}(dx)$ -- the Pettis integral of $(\log_\mu)_{\#}\mathbb{Q}$ -- is the unique element of $L_\mu\X$ that satisfies
\[\forall h\in L_\mu\X\,:\,\left\langle\int_{\X}\log_\mu x\,\mathbb{Q}(dx),h\right\rangle_\mu = \int_{\X}\langle \log_\mu x,h\rangle_\mu\,\mathbb{Q}(dx)~.\]
Moreover, we have the ``first order optimality condition''
\begin{equation}\label{eq:Pzeromean}\forall b\in\X\,:\,\E\,\langle \log_\mu X,\log_\mu b\rangle_\mu = \int_\X\,\langle \log_\mu x,\log_\mu b\rangle_\mu\,\mathbb{P}(dx) = 0~.\end{equation}
\end{proposition}

\begin{proof} This is essentially a rewording of  \cite[Theorem 2.4]{LePaRiSt22}. The existence and properties of $L_\mu\X$ are stated in that theorem, and its item 2 gives:
\begin{equation}\label{eq:logpintegralis0v1}\int_{\X}\int_{\X}\,\langle \log_\mu x,\log_\mu y\rangle_\mu \,\mathbb{P}(dx)\,\mathbb{P}(dy)=0~.\end{equation}
The identity involving the Pettis integral is item 4 in the same theorem. Combining  this identity with (\ref{eq:logpintegralis0v1}) finishes the proof.\end{proof}

Before continuing, we notice that (\ref{eq:huggingalmostconv}) and (\ref{eq:Pzeromean}) together imply that
\begin{align*}
\EXP \left( d^2(b,X)-d^2(\mu,X) \right)
&= k_{\min} d^2(b,\mu)  >0
\end{align*}
whenever $b\neq\mu$. Therefore, the Fr\'{e}chet mean is unique. 

For the remainder of the proof, we let $\wh{\mu}:=\wh{\mu}^\epsilon_{n,t}$ be the estimator from Section \ref{sub:definetrimmed}.
We fix a countable dense subset $\X_{\rm cd}\subset\X$ and define
\begin{align*}
\cE &:= \sup_{b\in \X_{\rm cd } \, : \, d(b,\mu)>0}\frac{1}{d(b,\mu)} \wh{T}_{n,t}^\epsilon\left(\langle \log_\mu(\cdot), \log_\mu(b) \rangle_\mu\right)~.
	\end{align*}

\begin{lemma}
\label{lem:bound.by.calE}
	We have that
    \begin{align}\label{eq:distancetrimmed}
    d(\wh{\mu},\mu) & \leq   \frac{\sqrt{2}+1}{k_{\min}} \cE~.
    \end{align}
\end{lemma}
\begin{proof}%
	{\em Step 1:} 	We  claim that  for all $b\in\X$,
\begin{align}
\label{eq:lower.bound.T}
\wh{T}_{n,t}^\epsilon\left( d^2(b,\cdot) -    d^2(\mu,\cdot)\right) 
  \ge& \left( d(b,\mu)\sqrt{k_{\min}} - \frac{\cE}{\sqrt{k_{\min}}}\right)^2
    - \frac{\cE^2}{k_{\min}}~.
\end{align}
	Indeed, first note that  it suffices to prove \eqref{eq:lower.bound.T} only for all $b\in \X_{\rm cd}$ (by continuity).
	Fix such $b$.
	By (\ref{eq:huggingalmostconv}), the monotonicity of $\wh{T}_{n,t}^\epsilon(\cdot)$ and the definition of $\mathcal{E}$,
\begin{eqnarray*}
   \wh{T}_{n,t}^\epsilon\left( d^2(b,\cdot) -  d^2(\mu,\cdot)\right)   
  &  \ge  &   k_{\min} d^2(b,\mu) - 2d(b,\mu)
\frac{1}{d(b,\mu)} \wh{T}_{n,t}^\epsilon \left(\langle \log_\mu(\cdot), \log_\mu(b) \rangle_\mu\right) \\
&\geq &    k_{\min} d^2(b,\mu) - 2d(b,\mu) \cE \\
&= &\left( d(b,\mu)\sqrt{k_{\min}} - \frac{\cE}{\sqrt{k_{\min}}}\right)^2
    - \frac{\cE^2}{k_{\min}}~.
\end{eqnarray*}
and \eqref{eq:lower.bound.T} follows.

{\em Step 2:} We apply Lemma \ref{lem:localization}.
To that end, define 
\[r:=\frac{\sqrt{2}+1}{k_{\min}} \cE\quad \mbox{ and }\quad \gamma :=\frac{\cE^2}{k_{\min}}\]
and observe that by \eqref{eq:lower.bound.T}, $\wh{T}_{n,t}^\epsilon( d^2(b,\cdot) -    d^2(\mu,\cdot)) \geq -\gamma$ all $b\in \mathcal{X}$.
Moreover, if $b\in\mathcal{X}$ satisfies that $d(b,\mu)>r$, then \eqref{eq:lower.bound.T} implies that $\wh{T}_{n,t}^\epsilon( d^2(b,\cdot) -   d^2(\mu,\cdot)) >\gamma$.
Therefore, an application of Lemma \ref{lem:localization} shows that $   d(\wh{\mu},\mu) \leq r$,  as claimed.
\end{proof}
      
With Lemma \ref{lem:bound.by.calE} at hand, the next ingredient in the proof of Theorem \ref{thm:geodesic.main} consists in accurately bounding $\mathcal{E}$.
To that end, we will apply Theorem \ref{thm:trimmed.mean} to the class of functions
\[\mathcal{F} :=\left\{ \frac{\langle \log_\mu(\cdot), \log_\mu(b) \rangle_\mu }{d(b,\mu)}  \, : \,  b\in\mathcal{X}_{\rm cd} \text{ and } d(b,\mu)>0\right\}\]
so that 
\[ \cE = \sup_{f\in\mathcal{F}} \wh{T}_{n,t}^\epsilon(f)~.\]
Observe that by the first-order condition for $\mu$, see \eqref{eq:Pzeromean},  each function in $\mathcal{F}$ has zero mean.
In order to control the ``global variance''
\[ V_n:=\EXP  \sup_{b\in \X_{\rm cd} \,: \, d(b,\mu)>0}  \left| \frac{1}{n d(b,\mu)} \sum_{i=1}^n
    \langle\log_\mu(X_i), \log_\mu(b) \rangle_\mu \right| \]
which appears in Theorem \ref{thm:trimmed.mean}, let $\mathbb{P}_n$ denote the random empirical measure of the points $X_1,\ldots,X_n$, which puts mass $1/n$ of each of them. Since the support of $\mathbb{P}_n$ is almost surely contained in ${\rm supp}(\mathbb{P})$, we can apply (\ref{eq:Pettisuse}) in Proposition \ref{prop:stolenfromLePaRiSt} and the Cauchy-Schwarz-type inequality (\ref{eq:CSinthecone}) to obtain
\begin{align}
\label{eq:VnHilbert}
V_n
&:=\EXP  \sup_{b\in \X_{\rm cd} \,: \, d(b,\mu)>0}  \left| \frac{1}{d(b,\mu)} 
   \left\langle\frac{1}{n}\sum_{i=1}^n\log_\mu(X_i), \log_\mu(b) \right\rangle_\mu \right|
\nonumber\\
&\leq  \EXP\,\left\|\frac{1}{n}\sum_{i=1}^n\log_\mu(X_i)-o_\mu\right\|_\mu~.
\end{align}

\begin{lemma}
\label{lem:variance.sum}
	We have that 
	\[ V_n
    \leq \sqrt\frac{ \E d^2(X,\mu)}{n}.\]
\end{lemma}
\begin{proof}
	It follows from Proposition \ref{prop:stolenfromLePaRiSt} that the random vectors $(\log_\mu(X_i))_{i=1}^n$ are i.i.d.\ mean-zero random elements of the Hilbert space $L_\mu\X$, and $o_\mu$ is the zero element of this space. Therefore, by Jensen's inequality and by the law of sums of variances in a Hilbert space,
\begin{align*}
V_n^2 
&\leq \EXP\left\| \frac{1}{n} \sum_{i=1}^n\log_\mu(X_i) - o_\mu\right\|_\mu^2 
= \frac{ \EXP\|\log_\mu(X) - o_\mu\|_\mu^2 }{n}
= \frac{ \EXP d^2(X,\mu) }{n}~,
\end{align*}
where the last equality follows from (\ref{eq:logpdistance}).\end{proof}

We are now ready for the proof of Theorem \ref{thm:geodesic.main}.

\begin{proof}
	By Lemma \ref{lem:bound.by.calE}, we have that 
	\[ d(\wh{\mu},\mu)\leq \frac{\sqrt{2}+1}{k_{\min}} \cE~. \]
	Moreover, Theorem \ref{thm:trimmed.mean} implies that, with probability at least $1-\delta$, 
	\begin{align*}
	\cE 
	&\leq C_\epsilon \left( 8V_n +  \inf_{p\in[1,2]} \nu_p\,\left( \frac{\ln(3/\delta)}{n}\right)^{1- \frac{1}{p}} + \inf_{p\in[1,\infty)}\nu_p\,\epsilon^{1-\frac{1}{p}}   \right) \\
	&\leq C_\varepsilon \left( 8 \sqrt \frac{ \E d^2(X,\mu)}{ n}  + \sqrt \frac{\sigma^2_{\rm w} \log(3/\delta)}{n} + \inf_{p\in[1,\infty)}\nu_p\,\epsilon^{1-\frac{1}{p}}  \right)~,
	\end{align*}
	where the second inequality follows from Lemma \ref{lem:variance.sum} and by choosing $p=2$ in the second term (noting that $\nu_2=\sigma_{\rm w}$ by definition). $C_\epsilon$ is defined in Theorem \ref{thm:trimmed.mean}.
	This completes the proof.   
\end{proof}

\subsection{Proof of Theorem \ref{thm:geodesic.subgaussian}}\label{proof:geodesic.subgaussian}

We present a proof that follows closely the proof of Theorem \ref{thm:geodesic.main}.
To that end, note that $\wh{T}_{n,0}^0(f)=\frac{1}{n}\sum_{i=1}^n f(X_i)$ is the empirical mean of  $f$ computed using the (uncorrupted) sample, hence 
\begin{align*}
\overline{\mu}_n = \wh{\mu}^0_{n,0} 
\in \mathop{\rm argmin}_{b\in\X}  \frac{1}{n}\sum_{i=1}^n d^2(X_i,b)
&= \mathop{\rm argmin}_{b\in\X} \sup_{a\in\X}  \frac{1}{n}\sum_{i=1}^n \left( d^2(b,X_i)- d^2(a,X_i) \right)  \\
&\equiv \mathop{\rm argmin}_{b\in\X} \sup_{a\in\X} \wh{T}_{n,0}^0\left( d^2(b,\cdot)- d^2(a,\cdot) \right)~.
\end{align*}
Let $\X_{\rm cd}\subset\X$ be a countable dense subset and set
\[ \cE := \sup_{b\in \X_{\rm cd } \, : \, d(b,\mu)>0}\frac{1}{ nd(b,\mu)} \sum_{i=1}^n  \langle \log_\mu(X_i), \log_\mu(b) \rangle_\mu~.\]
Then the same proof as given for Lemma \ref{lem:bound.by.calE} shows that
    \begin{align}
    \label{lem:bound.by.calE.gaussian}
    d\left(\overline{\mu}_n,\mu \right) 
    & \leq   \frac{\sqrt{2}+1}{k_{\min}} \cE~.
    \end{align}
Thus, it suffices to estimate $\cE$.

To this end, an application of Proposition \ref{prop:stolenfromLePaRiSt} in the spirit of (\ref{eq:VnHilbert}) shows that
\begin{align}
\label{eq:cE.gaussian}
\cE
&\leq  \left\|\frac{1}{n}\sum_{i=1}^n  \log_\mu(X_i) - o_\mu\right\|_\mu.
    \end{align}
Moreover, we recall from the same proposition that $ \log_\mu(X)$ is a random vector taking its values in the Hilbert space $L_\mu\X\subset T_\mu\X$, and that $\|\cdot\|_\mu$ is the Hilbert-norm on $L_\mu\X$.
Thus, by standard arguments (for example, a sub-Gaussian chaining argument combined with Talagrand's celebrated majorizing measures theorem see, e.g., Theorems 2.7.13 and 2.10.1 in \cite{talagrand2022upper}), it follows that there exists an absolute constant $C$ such that, with probability at least $1-\delta$,
\begin{align}
\label{eq:gaussian.estimate}
\left\|\frac{1}{n}\sum_{i=1}^n  \log_\mu(X_i) - o_\mu \right\|_\mu
\leq CL  \left( \frac{ \E \| Z\|_\mu }{ \sqrt n} + \sqrt \frac{ \sigma_{\rm w}^2 \log(1/\delta) }{n}   \right)~.
\end{align}
Here $Z$ is a zero mean Gaussian random variable in the Hilbert space $L_\mu\X$ with the same covariance as $\log_\mu(X)$.
In particular,
\begin{align*}
 \E \|Z\|_\mu^2 
&= {\rm trace} \, {\rm \mathbb{C}ov}[Z] 
= {\rm trace} \,  {\rm \mathbb{C}ov}[\log_\mu(X)]\\
&=  \E \|\log_\mu(X) - o_\mu\|_\mu^2
= \E d^2(X,\mu)~, 
\end{align*}
where the last equality follows from (\ref{eq:logpdistance}). Hence, by Jensen's inequality, $\E \|Z\|_\mu \leq \sqrt{ \E d^2(X,\mu)}$.
Thus the proof of Theorem is completed by \eqref{lem:bound.by.calE.gaussian}, \eqref{eq:cE.gaussian} and \eqref{eq:gaussian.estimate}. 
\endproof

\section{Fr\'{e}chet mean estimation in uniformly convex Banach spaces}
\label{sec:banach.space}

In this section, we show that for uniformly convex Banach spaces, the methods developed in this paper remain applicable, allowing us to construct statistically well-behaved estimators for the Fréchet mean.
Let $(\X,\|\cdot\|)$ be a separable Banach space and assume that it is {\em uniformly convex of power type $2$}: there exists a constant $c_{\X}>0$ such that for all $\eta\in (0,2]$ and $x,y\in \X$ with $\|x\|=\|y\|=1$,
\[\|x-y\|\geq \eta \hspace{1em} \Rightarrow  \hspace{1em} \textbf{}\left\|\frac{x+y}{2}\right\|\leq 1-c_\X\eta^2~.\]
The spaces $\ell^p$ with $p\in(1,2]$ are uniformly convex of power type 2. We discuss in detail these spaces in Sections \ref{sec:example.lp} and \ref{sec:frechet.vs.mean}.
Denote by  $(\X^\ast,\|\cdot\|_\ast)$ the dual of $(\X,\|\cdot\|)$, by $\langle \cdot,\cdot\rangle\colon\X^\ast\times\X\to\R$ the natural pairing of $\X^\ast$ and $\X$, and fix a Borel-measurable function
\[ g\colon \X \to\X^\ast, \qquad g(x)\in\partial\left( \|\cdot\|^2 \right)(x)~,\]
where $\partial f$ denotes the subgradient of a convex function $f$.
The existence is guaranteed by Theorem \ref{thm:measurable.subgradient}.
Note that if $f$ is differentiable, then $\partial f =\{ \nabla f\}$. In particular, $\partial \|\cdot\|^2(x)=\{2x\}$ if $\|\cdot\|$ is the Euclidean norm.

Moreover, assume that $\EXP \|X\|^2<\infty$. 
We let $\mu\in\mathop{\rm argmin}_{b\in\X} \E\|X-b\|^2 $
be the unique Fréchet mean of $X$:
existence of $\mu$ is guaranteed by Lemma \ref{lem:existence.barycenter}, whereas uniqueness follows from Remark \ref{rem:uniquefrechetbanach} below.
Finally, set
\[ \nu_p:=\sup_{x\in \X \, : \|x\|=1} \left( \E |\langle  g(X-\mu), x\rangle|^p \right)^{1/p}\]
and put $\sigma^2_{\rm w}:= \nu_2^2$.
Recall that $(X_i^\epsilon)_{i=1}^n$ is the $\varepsilon$-corrupted sample.
The following is the main result of this section.

\begin{theorem}
\label{thm:main.norm.space}
	Let $\delta\in(0,1)$, assume that
	\[t:=\lfloor \epsilon n\rfloor +\lceil \ln(2/\delta)\rceil \vee \left\lceil \tfrac{\epsilon\wedge \left(\frac{1}{2} - \epsilon\right)}{2}\,n\right\rceil < \frac{n}{2}~,\]
	and that $(\X,\|\cdot\|)$ is uniformly convex of power type $2$ with constant $c_\X$.
	Then the estimator $\wh{\mu}=\wh{\mu}_{n,t}^\varepsilon$ from Section \ref{sub:trimmedestimators} satisfies that with probability at least $1-\delta$,
\[ \left\|\wh{\mu}- \mu\right\|
\leq \frac{C_\eps }{C_\X} \left( 23\sqrt \frac{  \E \|X-\mu\|^2 }{ C_\X n} +  \sqrt \frac{ \sigma_{\rm w}^2 \log(3/\delta)}{n} + \inf_{p\in[1,\infty)} \nu_p \varepsilon^{1-1/p} \right)~.\]
	where $C_\epsilon=928(1+\frac{\epsilon}{\epsilon\wedge (\frac{1}{2}-\epsilon)})$ and $C_\X$ is a constant that only depends on $c_\X$.
\end{theorem}

The proof of the theorem is given in Section \ref{sec:proofof.main.norm.space} below.

Analogous to the explanations detailed in Remark \ref{rem:breakdown.point.geodesic}, the breakdown point of the estimator constructed in Theorem \ref{thm:main.norm.space} is $\varepsilon\approx \frac{1}{2}$.
In Section \ref{sec:example.lp} we will apply the results of the theorem to $\ell^p$-spaces, for which we also explicitly compute the constants $C_\X$. 
Section \ref{sec:frechet.vs.mean} discusses aspects concerning the estimator's optimality and non-optimality.

Finally, we establish high-probability estimates for the standard empirical barycenter $\overline{\mu}_n$ in the case of light-tailed (and uncorrupted) data.
To that end, we say that a random vector $X$ with Fréchet mean $\mu$ is $L$-sub-Gaussian if for all $x\in \X$ and $u\geq 0$,
\[ \mathbb{P}\left( |\langle  g(X-\mu), x \rangle | \geq u \|\langle g(X-\mu) ,x  \rangle \|_{L_2}  \right)
\leq 2\exp\left(-\frac{u^2}{2 L^2} \right)~. \]
In order to avoid digressing into technicalities, we assume here that $\mathcal{X} = \mathbb{R}^m$, that $\|\cdot\|^2$ is differentiable, and that its gradient $\nabla(\|\cdot\|^2)$ is invertible with a measurable inverse (see Remark \ref{rem:sub.gaussian.Banach.general} for a discussion on the result in the general case).
Those assumptions are satisfied, for instance, when $\X=\ell^p$ and $p\in(1,2]$, see Section \ref{sec:example.lp}.

\begin{theorem}
\label{thm:norm.space.subgaussian}
    Let $\delta\in(0,1)$ satisfy that $\delta\geq \exp(-n)$.
    Then, with probability at least $1-\delta$,	
\[ \left\|\overline{\mu}_n- \mu\right\|
\leq \frac{CL}{C_\X} \left(  \sqrt \frac{ \E \|Y-\mu\|^2 }{  n} +  \sqrt \frac{ \sigma_{\rm w}^2 \log(3/\delta)}{n} \right)~.\]
    Here $C$ is an absolute constant and $Y$ is a random vector for which $g(Y-\mu)$ has Gaussian distribution with mean 0 and the same covariance as $g(X-\mu)$, and, as before,  $C_\X$ is a constant that only depends on $c_\X$.
\end{theorem}

The theorem is proved in Section \ref{sec:main.norm.space.subgaussian}.

\begin{remark}
\label{rem:sub.gaussian.Banach.general}
    In the general case (when $\|\cdot\|^2$ is not necessarily differentiable with an invertible gradient) the assertion of Theorem \ref{thm:norm.space.subgaussian} remains valid, but the term $\E \|Y-\mu\|^2$ needs to be replaced by $\E \|Z\|_\ast^2$, where $Z$ is a mean-zero Gaussian random vector with the same covariance matrix as $g(X-\mu)$.
\end{remark}

\subsection{Proof of Theorem \ref{thm:main.norm.space}}
\label{sec:proofof.main.norm.space}

We first observe that a separable uniformly convex space is necessarily reflexive by the Milman–Pettis theorem.
Moreover, the following equivalent definition of uniform convexity turns out to be more useful for our purposes: 
 the function $f(\cdot)=\|\cdot\|^2$ is convex and, moreover, for any $x,y\in \X$ and any subgradient $x^\ast\in\partial f(x)$ of $f$ at $x$,
\begin{equation}\label{eq:uniformlyconvex}
\|y\|^2 \geq \|x\|^2 + \langle x^\ast,y-x\rangle + \frac{C_{\X}}{2}\|x-y\|^2
\end{equation}
for a constant $C_{\X}$ depending only on $c_\X$, see, e.g.,
Borwein, Guirao, Hájek, and Vanderwerff \cite[Theorem 2.3]{borwein2009uniformly}.

For the remainder of this section, set $\widehat{\mu}:=\wh{\mu}_{n,t}^\varepsilon$ to be the trimmed-mean Fréchet-mean estimator defined in Section \ref{sub:trimmedestimators}.
Let $\X_{\rm cd} \subset \X$ be a countable dense subset and put 
\[\X_{\rm cd,1}:=\left\{ \frac{ x-\mu}{\|x-\mu\|}  : \, x\in\X_{\rm cd}, \, x\neq \mu\right\}~.\]

\begin{lemma}
\label{lem:norm.calE}
We have 
\[\|\wh{\mu}-\mu\|\leq \frac{\sqrt{2}+1}{C_\X}\cE~, 
\qquad \text{where }
 \cE := \sup_{x\in \X_{\rm cd,1} }\wh{T}_{n,t}^\epsilon\left( \langle g(\, \cdot \, -\mu) ,x \rangle \right)~.\]
\end{lemma}

The proof of Lemma \ref{lem:norm.calE} follows a similar reasoning as used in the proof of Lemma \ref{lem:bound.by.calE} and we only sketch it.

\begin{proof}
By the uniform convexity condition in \eqref{eq:uniformlyconvex}, for every  $b,x\in\mathcal{X}$, 
\begin{equation}\label{eq:strongconvexity2}
  \|b-x \|^2 - \|\mu-x \|^2  \ge   \langle  g(x-\mu), b-\mu\rangle  +    \frac{C_\X}{2} \|b-\mu\|^2~.
\end{equation}
Therefore,  using the monotonicity of $\wh{T}_{n,t}^\epsilon(\cdot)$ and the definition of $\mathcal{E}$, for every $b\in\X_{\rm cd}$,
\begin{eqnarray*}
  \wh{T}_{n,t}^\epsilon\left( \|b-\cdot \|^2 -    \|\mu-\cdot \|^2\right) 
  &  \ge  &   \frac{C_\X}{2} \|b-\mu\|^2 - \|b-\mu\| \, \wh{T}_{n,t}^\varepsilon \left(\left\langle  g(\cdot -\mu) , \frac{b-\mu }{\|b-\mu\|}\right\rangle\right) \\
&\geq &    \frac{C_\X}{2} \|b-\mu\|^2 - \|b-\mu\| \cE \\
&= &\frac{1}{2}\left( \left( \|b-\mu\|\sqrt{C_\X} - \frac{\cE}{\sqrt{ C_\X}}\right)^2
    - \frac{\cE^2}{C_\X} \right)~.
\end{eqnarray*}
The proof is completed by  Lemma \ref{lem:localization} applied with $r:=\frac{\sqrt{2}+1}{C_\X}\cE$ and $\gamma :=\frac{\cE^2}{2C_\X}$.
\end{proof}

\begin{lemma}
\label{lem:norm.variance.sum}
\[V_n
:= \EXP\sup_{ x\in \X_{\rm cd,1} }\left|\left\langle  \frac{1}{n}\sum_{i=1}^n g(X_i-\mu) ,x \right\rangle  \right|
\leq  \sqrt{\frac{ 8 \EXP\,\|X-\mu\|^2} { C_\X n}}~. \]
\end{lemma}

The proof of Lemma \ref{lem:norm.variance.sum} requires some preliminary facts.
First, we use that the dual space $(\X^*,\|\cdot\|_\ast)$ of a uniformly convex space is {\em uniformly smooth} of power type $2$:

\begin{lemma}
\label{lem:norm.sdual.smooth}
For every $x^\ast,y^\ast\in \X^\ast$ and $x\in \partial \|\cdot\|_\ast^2(x^\ast)$,
\begin{equation}
\label{eq:uniformlysmooth.new}
\| y^\ast \|_*^2 \leq \|x^\ast \|_*^2 + \langle y^\ast-x^\ast , x\rangle + \frac{2}{C_\X}\|y^\ast-x^\ast\|_*^2~.
\end{equation}
\end{lemma}

While Lemma \ref{lem:norm.sdual.smooth} is likely known, (see, for example, Zǎlinescu \cite{zalinescu1983uniformly}), we have not been able to find a reference that explicitly quantifies the dependence on $C_\X$. 
Thus, for the convenience of the reader, we provide the  proof:

\begin{proof}
Set $f:=\|\cdot\|^2$ so that by \eqref{eq:uniformlyconvex}, for every $x,y\in\X$ and $x^\ast\in\partial f(x)$,
\begin{align}
\label{eq:f.unif.convex}
f(y)\geq f(x) +\langle x^\ast, y-x\rangle + \frac{C_\X}{2} \|y-x\|^2~.
\end{align}
Moreover
$f^\ast(x^\ast):=\sup_{ x\in\X}(\langle x^\ast,x\rangle -f(x) ) = \frac{1}{4}\|x^\ast\|_\ast^2$ for every $x^\ast\in\X^\ast$.
Let  $x^\ast,y^\ast\in\X^\ast$ and $x\in \partial f^\ast(x^\ast)$.
Then  $x^\ast\in\partial f(x)$ and hence, by \eqref{eq:f.unif.convex},
\begin{align*}
f^\ast(y^\ast)
&=\sup_{y\in\X}\left(\langle y^\ast ,y\rangle - f(y) \right) \\
&\leq \sup_{y\in\X}\left(\langle y^\ast ,y\rangle - \left( f(x) + \langle x^\ast, y-x \rangle + \frac{C_\X}{2} \|y-x\|^2 \right) \right) \\
&=\sup_{y\in\X}\left(\langle y^\ast-x^\ast , y-x\rangle - \frac{C_\X}{2} \|y-x\|^2 \right) + \langle y^\ast, x \rangle -f(x)~. 
\end{align*}
A straightforward computation shows that
\[\sup_{y\in\X}\left(\langle y^\ast-x^\ast, y-x \rangle - \frac{C_\X}{2} \|y-x\|^2 \right) 
= \frac{ \|y^\ast-x^\ast\|_\ast^2 }{2C_\X}\]
and since $x\in\partial f^\ast(x^\ast)$  we have that 
\[\langle  y^\ast ,x \rangle -f(x)
=\langle  y^\ast-x^\ast, x \rangle + \langle  x^\ast ,x \rangle -f(x)
= \langle y^\ast-x^\ast , x \rangle + f^\ast(x^\ast)~.\]
Thus 
\[f^\ast(y^\ast)\leq f^\ast(x^\ast) + \langle  y^\ast-x^\ast , x \rangle + \frac{ \|y^\ast-x^\ast\|_\ast^2 }{2 C_\X}\]
from which the claim readily follows.
\end{proof}

\begin{lemma}
\label{lem:dual.norm.2}
	For every $x\in\X$, we have that $\|g(x-\mu)\|_*  = 2 \|x-\mu\|$.
\end{lemma}
\begin{proof}
	Set $f:=\|\cdot\|^2$ so that $f^\ast=\frac{1}{4}\|\cdot\|^2_\ast$ and fix $y\in\X$.
	For any $y^\ast\in \partial f(y)$, we have  $f(y) = \langle y^\ast ,y \rangle - f^\ast(y^\ast)$ and hence, by duality,
	\[ \frac{1}{4} \|y^\ast\|_\ast^2 
    \leq \|y^\ast\|_\ast \cdot \| y\| - \|y\|^2. \]
    Resorting terms implies that
 $( \frac{1}{2} \|y^\ast\|_\ast - \|y\|)^2\leq 0$,
    and thus $\|y^\ast\|_\ast = 2\|y\|$.
  	Applying this inequality to $y= x-\mu$ and $y^\ast=g(x-\mu)$ clearly implies the statement.
\end{proof}

We are now ready for the proof of Lemma \ref{lem:norm.variance.sum}.

\begin{proof}
Set $Z_i:=g(X_i-\mu)$ and $S_\ell :=\sum_{i=1}^\ell Z_i$ for $1\leq i,\ell\leq n$.
By the definitions of $\X_{\rm cd,1}$ and the dual norm $\|\cdot\|_\ast$,
\[\sup_{x\in \X_{\rm cd,1 }}\left\langle   \sum_{i=1}^n g(X_i-\mu) , x\right\rangle
\leq  \|S_n\|_\ast~.\]
Therefore, if we can show that 
\begin{align}
\label{eq:exp.S.n}
\EXP \|S_n\|_\ast^2 
\leq \frac{2}{C_\X} \sum_{i=1}^n \E \|Z_i\|_\ast^2 
= \frac{2 n \E \|Z_1\|_\ast^2 }{C_\X}~,
\end{align}
then the proof is completed by Jensen's inequality and the fact that $\E \|Z_1\|_\ast^2 =  4 \E \|X-\mu\|^2$ by Lemma \ref{lem:dual.norm.2}.

To prove \eqref{eq:exp.S.n}, an application of Lemma \ref{lem:norm.sdual.smooth} (with $y^\ast= S_n$ and $x^\ast=S_{n-1}$ and thus $y^\ast-x^\ast = Z_n$) shows that 
\begin{align}
    \label{eq:Sn.recursice}
    \|S_n\|_\ast^2 \leq \|S_{n-1}\|_\ast^2 + \langle Z_n ,x\rangle +  \frac{2}{C_\X} \|Z_n\|_\ast^2~,
\end{align} 
for any choice $x=x(S_{n-1})$  in the subgradient of $\|\cdot\|_\ast^2$ at $S_{n-1}$. 
Since  $\E \langle Z_n, y\rangle=0$ for all $y\in\X$ by the first order optimality condition of $\mu$  and since $x(S_{n-1})\in\partial \|\cdot\|_\ast^2(S_{n-1})$ can be chosen to be measurable, it follows from independence that $ \EXP[ \langle  Z_n, x(S_{n-1})\rangle | S_{n-1}]=0$ almost surely.
Therefore, taking expectations in \eqref{eq:Sn.recursice} and using the tower property,
\begin{align*}
\EXP \|S_n\|_\ast^2 
&\leq \EXP \|S_{n-1}\|_\ast^2 + \EXP\left[ \EXP\left[ \langle  Z_n , x(S_{n-1}) \rangle | S_{n-1} \right] \right] + \frac{2}{C_\X} \E \|Z_n\|_\ast^2 \\
&=\EXP \|S_{n-1}\|_\ast^2 +  \frac{2}{C_\X} \E \|Z_n\|_\ast^2~.
\end{align*}
The proof of \eqref{eq:exp.S.n} follows from an induction over $n$.
\end{proof}

\begin{remark}[Uniqueness of Fréchet mean]\label{rem:uniquefrechetbanach} Before proceeding, we observe that our arguments imply uniqueness of the Fréchet mean. Indeed, it suffices to combine the first-order optimality condition $\EXP[\langle g(X-\mu),y\rangle]=0$ for all $y\in\X$, noted in the preceding proof, with the strong convexity inequality (\ref{eq:strongconvexity2}).
\end{remark}

We are now ready for the proof of Theorem \ref{thm:main.norm.space}.

\begin{proof}
	We apply Theorem \ref{thm:trimmed.mean} to the class of functions
	\[ \mathcal{F} = \left\{ \langle g(\cdot - \mu), x  \rangle \, :\,  x\in \X_{\rm cd,1} \right\} \]
	so that $\cE = \sup_{f\in \mathcal{F} } \wh{T}_{n,t}^\epsilon (f)$ and $\|\wh{\mu}-\mu\|\leq \frac{\sqrt 2 +1 }{C_\X}\cE$ by Lemma \ref{lem:norm.calE}. 
	
	Each function in $\mathcal{F}$ has zero mean by the first order optimality condition of $\mu$ and is integrable (e.g., by Lemma \ref{lem:dual.norm.2}).
	Thus, by Theorem \ref{thm:trimmed.mean}, with probability at least $1-\delta$,
	\begin{align*}
	\mathcal{E}
\leq & C_\epsilon \left( 8V_n
+ \,\inf_{p\in[1,2]} \nu_p\,\left(\frac{\ln(3/\delta)}{n}\right)^{1- \frac{1}{p}} + \inf_{p\in[1,\infty)}\nu_p\,\epsilon^{1-\frac{1}{p}}\right) \\
\leq & C_\epsilon \left( 8\sqrt\frac{ 8 \EXP \|X-\mu\|^2 }{ C_\X n} + \sqrt\frac{\sigma^2_{\rm w} \log(3/\delta)}{n} + \inf_{p\in[1,\infty)}\nu_p\,\epsilon^{1-\frac{1}{p}}\right)~,
\end{align*}
	where the second inequality follows from Lemma \ref{lem:norm.variance.sum} and by choosing $p=2$, respectively.	
\end{proof}

\subsection{Proof of Theorem \ref{thm:norm.space.subgaussian}}
\label{sec:main.norm.space.subgaussian}

The argument proceeds along the same lines as the proof of Theorem \ref{thm:geodesic.subgaussian}.
First, observe that $\|\overline\mu_n-\mu\|\leq  \frac{\sqrt{2}+1}{C_\X} \mathcal{E}$, where 
\[\mathcal{E}= \sup_{x\in \X_{\rm cd,1}} \frac{1}{n} \sum_{i=1}^n \left\langle  g(X_i-\mu) , x\right\rangle
\leq \left\| \frac{1}{n}\sum_{i=1}^n  g(X_i-\mu) \right\|_\ast.\]
Next, since $g(X-\mu)$ is $L$-sub-Gaussian by assumption, the same arguments as used in the proof of Theorem \ref{thm:geodesic.subgaussian} imply that  with probability at least $1-\delta$,
\[
\left\| \frac{1}{n}\sum_{i=1}^n  g(X_i-\mu) \right\|_\ast
\leq  C L \left( \frac{\E \|Z\|_\ast }{\sqrt n} + \sqrt{ \frac{\sigma_{\rm w}^2 \log(3/\delta)}{n} }\right)~, \]
where $C$ is an absolute constant and $Z$ is a zero-mean Gaussian vector on $\X^\ast$ with the same covariance as $g(X-\mu)$.
Finally, fix any random vector  $Y$ for which $g(Y-\mu)$ has zero-mean Gaussian distribution with the same covariance as $g(X-\mu)$ (such a vector $Y$ exists because $g$ is invertible, e.g.\ we may set $Y=g^{-1}(Z)+\mu$).
By Lemma \ref{lem:dual.norm.2}, 
\[4 \E \|Y-\mu\|^2
=\E \| g(Y-\mu)\|_\ast^2
=\E \|Z\|_\ast^2~,\]
from which the proof readily follows.
\endproof

\section{Examples}
\label{sec:examples}

In this section, we discuss the main results of the paper by pointing out the necessity of their conditions and by applying them to concrete examples, including Wasserstein spaces.

\subsection{A simple example showing that some conditions are needed}

We start with an example of a simple space $\X$ with a metric $d$ where one cannot hope to obtain any nontrivial estimation rate, even under assumptions that guarantee that the Fr\'{e}chet mean is unique and the data generating distribution is ``nice.''

Let $\X=\R$ with the metric $d(x,y):=\sqrt{|x-y|}$ ($x,y\in \R$). If $X$ is a scalar random variable  on $\R$ such that $\EXP|X|<+\infty$, then the Fr\'{e}chet means of $X$ for this specific choice of metric are precisely the medians of $X$. 

The estimation of medians is a classical topic in mathematical statistics. In particular, the sample median of an i.i.d.\ sample satisfies a central limit theorem around the population median $\mu$ whenever the distribution of $X$ has a continuous, strictly positive density around $\mu$, see Walker  \cite{walker1968quantiles}. It turns out that, in the absence of such strong conditions, there can be no finite-sample bound for estimating medians, as the following example shows. 

Let $\mathbb{Q}$ be a probability distribution with a smooth density that is symmetrical around $0$ and is supported over $\R\setminus(-1,1)$. Now, given $\mu\in(-1,1)$ and $\eta\in (0,1)$, let $\mathbb{P}_{\mu,\eta}$ denote the  mixture distribution
\[\mathbb{P}_{\mu,\eta} = (1 - \eta)\,\mathbb{Q} + \eta\, \mathcal{N}(\mu,1)~,\]
were $\mathcal{N}(\mu,1)$ denotes the normal distribution with mean $\mu$ and variance 1.
One can check that the (unique) median of $\mathbb{P}_{\mu,\eta}$ is $\mu$, and also that $\mathbb{P}_{\mu,\eta}$ has a positive, smooth density at all points. However, for a sample size $n$ and $0<\eta n\ll 1$, an i.i.d.\ size-$n$ sample from $\mathbb{P}_{\mu,\eta}$ is close in total variation to a sample from $\mathbb{Q}$. In other words, the sample carries essentially no information about $\mu$.

\subsection{Geodesic space: the Wasserstein space}
\label{sec:example.wasserstein} 

We follow the notation used in \cite{LePaRiSt22}.
Let $H$ be a separable Hilbert space and let $(\X,d)=(\mathcal{P}_2(H),\mathcal{W}_2)$ be the Wasserstein space over $H$, where $\mathcal{P}_2(H)$ is the set of probability measures on $H$ with finite second moment and
\begin{align}
\label{eq:def.W2}
 \mathcal{W}_2(\gamma,\nu)=\inf_\pi \left( \int_{H\times H} \|x-y\|_H^2 \,\pi(dx,dy) \right)^{1/2}
\end{align}
where the infimum is taken over all couplings (or randomized transport plans) $\pi$, that is, probability measures with first marginal $\gamma$ and second marginal $\nu$.
We refer to Ambrosio, Gigli, and Savaré \cite{ambrosio2008gradient}
or Villani \cite{villani2008optimal} for background on optimal transport.
The space $(\mathcal{P}_2(H),\mathcal{W}_2)$ has non-negative curvature and a  tractable description of tangent cones.
Indeed, assuming for the remainder of this paragraph that $\gamma$ is regular\footnote{For instance, when  $H=\R^m$ and $\gamma$ is absolutely continuous with respect to the Lebesgue measure.}, then by Brenier's theorem, the infimum in \eqref{eq:def.W2} is attained by a deterministic transport plan $\pi$ of the form $\pi={\rm Law}( Z, \nabla\varphi_{\gamma\to\nu}(Z)))$ where $Z\sim \gamma$ and $\varphi_{\gamma\to\nu} \colon H\to \R$ is a convex function.
In particular 
\[\mathcal{W}_2^2(\gamma,\nu)=\int_H \|\nabla\varphi_{\gamma\to\nu}(z) -z\|_H^2\,\gamma(dz)~.\]
Moreover, $T_\gamma \mathcal{P}_2(H) \subset L_2(H,\|\cdot\|_H,\gamma)$ can be identified with the $L_2(\gamma)$-closure of gradients of functions, and  $\log_\gamma(\nu) = \nabla\varphi_{\gamma\to\nu} - {\rm Id}$.
Thus, if the Fréchet mean $\mu$ of a $(\mathcal{P}_2(H),\mathcal{W}_2)$-valued random vector $X$ is regular,  then
\begin{align*}
    \sigma_{\rm w}^2 
    &= \sup_{  \substack{  \psi \colon H \to H \text{ gradient} \\ \text{of a function, } \|\psi\|_{L_2(\mu)}\leq 1  }}  \E \left( \int_{\R^m} \langle \nabla\varphi_{\mu\to X}(z) - z, \psi(z) \rangle_{H} \, \mu(dz) \right)^2
\end{align*}
and 
\begin{align*}
    \E \mathcal{W}^2_2(X,\mu)
     &= \E  \sup_{  \substack{  \psi \colon H \to H \text{ gradient} \\ \text{of a function, } \|\psi\|_{L_2(\mu)\leq 1}  }}   \left( \int_{\R^m} \langle \nabla\varphi_{\mu\to X}(z) -z, \psi(z) \rangle_{H} \, \mu(dz) \right)^2  \\
    &=\E  \int_{\R^m} \| \nabla\varphi_{\mu\to X}(z) - z\|_H^2 \,\mu(dz)~.
\end{align*}

Next, recall that by Theorem \ref{thm:hugging.nonnegative}, lower bounds on the hugging function can be established through the extendibility of geodesic curves.
In the present setting, the latter turns out to be intrinsically linked to the convexity of optimal transport maps.
Recall that a function $\varphi\colon H\to \R$ is called $\alpha$-strongly convex if for all $x\in H$ the subgradient $\partial\varphi(x)$ is nonempty and for all $g\in \partial \varphi(x)$ and $y\in H$,
\[ \inr{g,x-y} \geq \varphi(x)-\varphi(y) + \frac{\alpha}{2} \|x-y\|_H^2~. \]
Moreover, $\varphi$ is $\beta$-smooth if for all $g\in \partial \varphi(x)$ and $y\in H$,
\[ \inr{g,x-y} \leq \varphi(x)-\varphi(y) + \frac{\beta}{2} \|x-y\|_H^2~.\]
Note that if $\beta>0$, then a $\beta$-smooth function must automatically be differentiable.
For intuition, observe that if $\varphi$ is twice differentiable, then it is $\alpha$-strongly convex and $\beta$-smooth if and only if all eigenvalues of its Hessian lie in the interval $[\alpha, \beta]$.

The following result follows from \cite[Theorem 3.5]{AhLePa20}; see also \cite[Section 4.2]{LePaRiSt22}.

\begin{theorem}[\cite{AhLePa20}]
\label{thm:hugging.wasserstein}
	Let $\gamma,\nu \in \mathcal{P}_2(H)$ be such that $\nu$ is the push-forward of $\gamma$ by the gradient of a convex function $\varphi$ and let $\lambda_{\rm in}, \lambda_{\rm out}>0$.
    Then the following are equivalent.
    \begin{itemize}
        \item The geodesic connecting $\gamma$ and $\nu$ is  $(\lambda_{\rm in}, \lambda_{\rm out})$-extendible.
        \item $\varphi$ is $\frac{\lambda_{\rm out} }{1+\lambda_{\rm out}}$-strongly convex and $\frac{1+\lambda_{\rm in}}{\lambda_{\rm in}}$-smooth.
    \end{itemize}
\end{theorem}

An immediate outcome of our main result (Theorem \ref{thm:geodesic.main}) combined with Theorem \ref{thm:hugging.wasserstein} and Theorem \ref{thm:hugging.nonnegative} is the following:

\begin{corollary}
\label{cor:Wasserstein}
	Let $X \sim \mathbb{P}$ be a square integrable random vector taking its values in $(\mathcal{P}_2(H),\mathcal{W}_2)$. 
    Let $\mu$ be a Fréchet mean of $X$. 
    Further suppose that every $\nu\in{\rm supp}(\mathbb{P})$ is the push-forward of $\mu$ by the gradient of a $\alpha$-strongly convex and $\beta$-smooth function. Then $\mu$ is the unique Fréchet mean of $X$.
    Let $\delta\in(0,1)$ and suppose that
 \[t=\lfloor \epsilon n\rfloor +\lceil \ln(2/\delta)\rceil \vee \left\lceil \tfrac{\epsilon\wedge \left(\frac{1}{2} - \epsilon\right)}{2}\,n\right\rceil < \frac{n}{2}~.\]
    Then  the trimmed-mean estimator $\widehat{\mu}=\wh{\mu}_{n,t}^\varepsilon$ defined in Section \ref{sub:trimmedestimators} satisfies, with probability at least $1-\delta$, 
	\[ \mathcal{W}_2(\widehat{\mu},\mu)
	\leq \frac{C_\varepsilon}{1+\alpha-\beta} \left( 8 \sqrt{\frac{ \EXP \mathcal{W}_2^2(X,\mu) }{n}}  + \sqrt{\frac{\sigma^2_{\rm w} \log(3/\delta)}{n}}
    + \sqrt{\epsilon}\sigma_{\rm w}
   \right)~,\]
    where $C_\varepsilon=928(1+\frac{\epsilon}{\epsilon\wedge (\frac{1}{2}-\epsilon)})$.
\end{corollary}

This result extends Corollary~4.4 in~\cite{LePaRiSt22}:
Both results are derived under the same set of assumptions, but Corollary~\ref{cor:Wasserstein} strengthens the conclusion in two key ways:  it provides a high-probability bound (as opposed to a constant-probability one), and it remains valid even in the presence of corrupted data.  
In particular, the convergence rate matches that of~\cite[Corollary 4.4]{LePaRiSt22}, up to a multiplicative constant.

\subsection{The Wasserstein space over the Gaussians}

Here we present a setting where the  general results detailed in the previous section become more explicit.
Let $H=\R^m$ be the Euclidean space and denote by $\mathcal{N}(a,V)$ the normal distribution with mean $a$ and covariance matrix $V$ in $\R^m$.
Throughout this section, we restrict to $V\in S^m_{++}$, the set of all positive-definite symmetric matrices.
We endow $S_{++}^m$ with the  Frobenius (or, Hilbert-Schmidt) inner product $\langle A, B\rangle :={\rm trace}(A^\top B)$ and associated norm  $\|\cdot\|_{\rm F}$, making it a Hilbert space.

For $\gamma=\mathcal{N}(a,V)$ to $\nu=\mathcal{N}(b,W)$, it is well-known that
\[\mathcal{W}_2^2(\gamma,\nu) = \|a-b\|_2^2 +  {\rm trace}\left( V + W + ( V^{1/2} W V^{1/2} )^{1/2}\right)  \]
and  that the optimal  transport mapping $\nabla \varphi_{\gamma\to\nu} \colon\R^m\to\R^m$ from $\gamma$ to $\nu$ is given by
\begin{align}
\label{eq:gaussian.transport}
\nabla \varphi_{\gamma\to\nu} (x) 
&=  V^{-1/2} ( V^{1/2} W V^{1/2} )^{1/2} V^{-1/2} (x-a)+ b~.
\end{align}
Let $X$ be a $(\mathcal{P}_2(\R^m),\mathcal{W}_2)$-valued random vector which takes its values almost surely in $\{ \mathcal{N}(a,V) : a\in\R^m ,\, V \in S_{++}^m\}$.
Thus we can identify $X$ with a pair of random vectors $({\bf a},{\bf V})$ taking values in $\R^m\times S_{++}^m$.
If  $ {\bf a}$ and ${\bf V}$ have finite second moments, then the Fréchet mean $\mu$ of $X$ exists and is given by $\mu=\mathcal{N}(a^\ast,V^\ast)$, where
\[ a^\ast=\E[{\bf a}] \quad \text{ and } \quad V^\ast \text{ solves } V^\ast = \E[ ( V^{\ast,1/2} {\bf V}  V^{\ast,1/2} )^{1/2}]~.\]
It follows that if all the eigenvalues of ${\bf V}$ lie in some interval $[\kappa_0,\kappa_1]$ almost surely, then the eigenvalues of $V^\ast$ must lie in that interval as well.
In particular, setting $\kappa=\kappa_1/\kappa_0$, it follows that $\nabla\varphi_{\mu\to\nu}$ is $1 /\kappa$ strongly convex and $\kappa$-smooth for almost every realization $\nu$ of $X$.

\begin{theorem}
    Let $\kappa_1,\kappa_0>0$ and let $X$ be a $(\mathcal{P}_2(\R^m),\mathcal{W}_2)$-valued random variable which takes its values almost surely in $\{ \mathcal{N}(a,V) : V \text{ has eigenvalues in } [\kappa_0,\kappa_1]\}$. 
	Set $\kappa=\kappa_1/\kappa_0$ and assume that $\kappa-\frac{1}{\kappa}<1$.
    Moreover, let $\delta\in(0,1)$ and assume that
 \[ t:=\lfloor \epsilon n\rfloor +\lceil \ln(2/\delta)\rceil \vee \left\lceil \tfrac{\epsilon\wedge \left(\frac{1}{2} - \epsilon\right)}{2}\,n\right\rceil < \frac{n}{2}~.\]
    Then the trimmed-mean estimator $\widehat{\mu}=\wh{\mu}_{n,t}^\varepsilon$ from Section \ref{sub:trimmedestimators} satisfies,  with probability at least $1-\delta$, 
	\[ \mathcal{W}_2(\widehat{\mu},\mu)
	\leq  \frac{C_\varepsilon}{1-\kappa+\frac{1}{\kappa}} \left( \sqrt\frac{	 \E \mathcal{W}_2^2(X,\mu) }{n} + \sqrt\frac{	\sigma^2_{\rm w} \log(3/\delta) }{n} 
     + \sqrt{\epsilon}\sigma_{\rm w}
 \right)~. \]
\end{theorem}

It is worthwhile noting the substantial difference between $\E\mathcal{W}_2^2(X,\mu) $  and $ \sigma^2_{\rm w}$ in this setting.
To see this, for simplicity, consider the case when  ${\bf V}={\rm diag}({\bf D})$ is almost surely diagonal.
In that case $V^{\ast,1/2} = \E [{\bf D}^{1/2}]$ and we
    have\footnote{Here we use the notation $a\sim b$ to indicate that there are two absolute constants $c,C>0$ for which $ca\leq b\leq Ca$.}
	\begin{align*}
	\E \mathcal{W}_2^2(X,\mu)
	&\sim {\rm trace}( \cov[{\bf a}] ) +  {\rm trace}\left( \cov[ {\bf D}^{1/2}] \right),  \\
 \sigma^2_{\rm w} 
	&\sim   \lambda_{\max} ( \cov[{\bf a}] ) + \lambda_{\max} \left( \cov[  {\bf D}^{1/2} ] \right) .
	\end{align*}
In particular, $\E \mathcal{W}_2^2(X,\mu)$ is dimension-dependent while $\sigma^2_{\rm w} $ is not.

\subsection{Fr\'{e}chet means for \texorpdfstring{$\ell^p$}{lp}  norms over \texorpdfstring{$\R^m$}{Rm}: general bounds}
\label{sec:example.lp}

A simple class of uniformly convex Banach spaces of power type $2$ are the spaces $\X=\R^m$ with the $\ell^p$ norm \[\|x\|_p:=\left(\sum_{\ell=1}^m|x(\ell)|^p\right)^{\frac{1}{p}}\, , \,\quad x=(x(\ell))_{\ell=1}^m\in\R^m,\] where $1<p\leq 2$ is fixed. We use $x(\ell)$ to denote the $\ell$-th coordinate of a vector $x\in\R^m$
 throughout this subsection.

The gradient of $\|\cdot\|_p^2$ at $x\in \R^m$ is the vector
\[g_p(x) = 2\,\|x\|^{2-p}_p \,\left( |x(\ell)|^{p-1}{\rm sign}(x(\ell)) \right)_{\ell=1}^m.\]
Moreover, we can compute the constant $C_{\X}$ explicitly for this example:

\begin{lemma}
The optimal strong convexity constant of $\|\cdot\|_p^2$ is $C_\X = 2(p-1)$. That is, for all $x,y\in\R^m$ there is $z_\ast \in\R^m$ that satisfies
\[
\|y\|_p^2 \geq \|x\|_p^2 + \langle z_\ast,y-x\rangle + (p-1)\|y-x\|^2_p~,\]
and the $(p-1)$ factor above cannot be replaced by any larger value.
\end{lemma}
\begin{proof} Let $q:=p/(p-1)\geq 2$ be the H\"{o}lder-dual exponent to $p$. It was shown, e.g., in \cite[Lemma 2.4]{dumbgen2010nemirovski} that the squared $\ell^q$ norm is smooth, in that it satisfies the following: 
for all  $x_\ast,y_\ast\in \R^m$ there is $ z\in \R^m$ such that
\[
\|y_\ast\|_q^2 \leq \|x_\ast \|_q^2 + \langle y_\ast-x_\ast,z\rangle + (q-1)\|y_\ast-x_\ast\|^2_q~.\]
Moreover, the constant $q-1$ cannot be improved. The reasoning in the proof of Lemma \ref{lem:norm.sdual.smooth} can then be used ``backwards'', with $2/C_\X=(q-1)$, to obtain the strong convexity constant
\[C_\X = \frac{2}{q-1} = \frac{2}{\frac{p}{p-1}-1} = 2(p-1)~,\]
which (by duality) cannot be improved.\end{proof}

It follows from the above that, given a $\varepsilon$-corrupted sample from a probability distribution $\mathbb{P}$ over $\R^m$ with finite second moment, and under the additional assumptions of Theorem \ref{thm:main.norm.space}, the estimator $\wh{\mu}=\wh{\mu}_{n,t}^\varepsilon(X_1^\epsilon,\ldots,X_n^\epsilon)$ from the theorem satisfies that, with probability at least $1-\delta$,
\begin{equation}\label{eq:lpboundexample} \left\|\wh{\mu}- \mu\right\|
\leq \frac{2C_\eps }{p-1} \left( 23\sqrt \frac{ 2\E \|X-\mu\|_p^2 }{(p-1)n} +  \sqrt \frac{ \sigma_{\rm w}^2 \log(3/\delta)}{n} + \inf_{\alpha\in[1,\infty)} \nu_\alpha \varepsilon^{1-1/\alpha} \right)~,\end{equation}
where
\[\nu_\alpha^\alpha:=\sup_{z\in\R^m\, : \, \|z\|_p=1} \EXP\,\left[\|X-\mu\|_p^{\alpha(2-p)}\left|2\sum_{\ell=1}^m |X(\ell)-\mu(\ell)|^{p-1}\,{\rm sign}(X(\ell)-\mu(\ell))\,z(\ell)\right|^\alpha\right]\]
and $\sigma^2_{\rm w} = \nu_2^2$.

For $p=2$ -- the case where the  Fr\'{e}chet mean is just the standard mean --, the bounds match those of standard mean estimation in \cite{LuMe20,oliveira2023trimmedsamplemeansrobust} up to the constant factors. On the other hand, the bounds for other $1<p<2$ are not as explicit, and the constants blow up when $p\to 1$.

\subsection{Fr\'{e}chet means versus standard expectations for \texorpdfstring{$\ell^p$}{lp}  norms}
\label{sec:frechet.vs.mean}

Here we briefly discuss the relationship of the bounds developed in this paper for the performance of Fréchet-mean estimators with those of the estimators of the ordinary mean.  
We consider the case of $\X=\R^m$ with the $\ell_p$ norm considered in the preceding subsection, but now we restrict our attention to the non-Euclidean case of $1<p<2$. We make the following simplifying assumption:

\begin{quotation}\noindent 
\emph{$X=\mu+Z$ where $\mu\in\R^m$ and $Z$ is a random vector in $\R^m$ with i.i.d.\ coordinates that are symmetric with variance 1 and finite moments of order $r$ for a sufficiently large constant $r$. 
Moreover, we fix $1<p<2$ and take $\varepsilon=0$.}\end{quotation}

This assumption guarantees that $\EXP[X-\mu]=\EXP[g_p(X-\mu)]=0$, so that $\mu$ is both the expectation and the Fr\'{e}chet mean of $X$. 
In particular, we may resort to known lower bounds for the accuracy when estimating the mean: 
by Depersin and Lecu\'{e} \cite{depersin2022optimal},  the best (mini-max) error that \emph{any} estimator $\widetilde{\mu}$ can achieve with confidence $1-\delta$ is 
\begin{align}
\label{eq:minimax.lower.any.mean.estimator}
     c\left( \frac{ \E \|Y-\mu\|_p }{ \sqrt n} + \sqrt{ \frac{ \sigma_{\rm w,m }^2 \log(3/\delta) }{n} } \right)~,
\end{align} 
where $Y$ is a Gaussian random vector with the same mean and covariance as $X$ and, with $q=\frac{p}{p-1}$ denoting the conjugate H\"older exponent to $p$,
\begin{align}
    \label{eq:mean.sigma}
    \sigma_{\rm w,m }^2 := \sup_{ z\in \R^m \, : \, \|z\|_q =1} \E \langle X-\mu,z\rangle^2.
\end{align}

\begin{remark}
    The result in \cite{depersin2022optimal} applies not only to $\|\cdot\|_p$ but to an arbitrary norm $\|\cdot\|$ on $\R^m$ (replacing $\|\cdot\|_p$ by $\|\cdot\|$ in \eqref{eq:minimax.lower.any.mean.estimator} and $\|\cdot\|_q$ by the dual norm $\|\cdot\|_\ast$ in \eqref{eq:mean.sigma}).
    An estimator that achieves  the minimax lower bound detailed in \eqref{eq:minimax.lower.any.mean.estimator} was recently constructed by Bartl and Mendelson \cite{bartl2025uniform} (for general norms), under the additional assumption that the statistician has access to a rough estimate on the covariance matrix of $X$. 
    Prior to that,  estimators were constructed in \cite{depersin2022optimal,LuMe18a} that `nearly' achieve the lower bound in \eqref{eq:minimax.lower.any.mean.estimator}, essentially with $\E \|Y-\mu\|$ therein replaced by
    \[  \E \left\| \frac{1}{\sqrt n}\sum_{i=1}^n (X_i-\mu)\right\|~.\]
\end{remark}

The minimax lower bound in \eqref{eq:minimax.lower.any.mean.estimator}  holds for any estimator, and in particular, it applies to the estimator constructed in Theorem \ref{thm:main.norm.space}, since in the setting of this section, the Fr\'echet mean coincides with the standard mean.
To compare the lower bound  in \eqref{eq:minimax.lower.any.mean.estimator} with the performance of our estimator, recall that by Theorem \ref{thm:main.norm.space}
\begin{align}
\label{eq:optimality.our.estimator}
    \|\widehat{\mu}-\mu \|_p \leq  C\left( \sqrt \frac{ \E \|X-\mu\|_p^2 }{n} + \sqrt{ \frac{ \sigma_{\rm w }^2 \log(3/\delta) }{n} } \right)~,
\end{align} 
where, using the explicit form of the subgradient from Section \ref{sec:example.lp},
\[ \sigma_{\rm w}^2=\sup_{z\in\R^m \, :\,\|z\|_p=1}4\EXP\,\left[\|X-\mu\|_p^{4-2p}\left|\sum_{\ell=1}^m |X(\ell)-\mu(\ell)|^{p-1}\,{\rm sign}(X(\ell)-\mu(\ell))\,z(\ell)\right|^2\right]~.\]

To compare \eqref{eq:optimality.our.estimator}  and \eqref{eq:minimax.lower.any.mean.estimator}, we first focus on the ``global variance'' terms therein, namely $\sqrt{\E\|X-\mu\|_p^2}$ and $\E \|Y-\mu\|_p$ (recall that $Y$ is a Gaussian random vector with the same covariance as $X$).
A straightforward calculation shows that both are of the same order and satisfy that
\[  \sqrt{\E\|X-\mu\|_p^2} \sim \E \|Y-\mu\|_p \sim m^{1/p}.\]
Thus, the global variance term in our estimator is optimal and cannot be improved further.

As for the local variances $\sigma_{\rm w}^2$ and $\sigma_{\rm 2, m}^2$, first note that 
\[ \sigma^2_{\rm w, m} = \sup_{z \in \R^m \, : \, \|z\|_q =1 } \langle \cov[X]z,z\rangle 
= \sup_{z \in \R^m \, : \, \|z\|_q =1  } \|z\|_2^2 
= m^{2/p-1}. \]
On the other hand, standard computations\footnote{Taking  $z=(1,0,\ldots,0)$ shows that $\sigma^2_{\rm w}\geq 4\E \|\widetilde{Z}\|_p^{4-2p}|Z(1)|^{2p-2}$ where $\widetilde{Z}=(0,Z(2),\ldots,Z(m))$ and we recall  that $Z=X-\mu$.
Further, $\widetilde{Z}$ and $Z(1)$ are independent, $\E|Z(1)|^{2p-2}\geq c_0$ (e.g.\ by the Paley-Zygmund inequality) and $\E \|\widetilde{Z}\|_p^{4/p-2}\sim m^{4-2p}$ as noted previously.
The corresponding upper bound on $\sigma^2_{\rm w}$ follows from H\"older's inequality.}
show that 
\[ \sigma^2_{\rm w} \sim m^{4/p-2}  =(\sigma^2_{\rm w,m})^2.\]
Thus the local variance term of our Fréchet mean estimator is suboptimal in this context. 
It remains unclear whether this suboptimality is an artifact of our proof or if it points to more fundamental differences between estimating expectations and Fréchet means. Obtaining minimax lower bounds for Fréchet mean estimation, akin to \eqref{eq:minimax.lower.any.mean.estimator}, is an interesting direction for future research.

\section{Concluding remarks}

In this paper, we constructed and analyzed a robust estimator of the Fréchet mean in general metric spaces. Our main results show that---generalizing a phenomenon known in Euclidean spaces---the estimator’s accuracy is governed by the sum of a confidence-independent ``global'' variance term and a smaller ``local'' variance term that depends on the confidence level.
These bounds hold in broad generality, including Alexandrov spaces with curvature bounded from below and uniformly convex Banach spaces. While they are known to be optimal in Euclidean spaces, it remains unclear when our upper bounds are optimal; nonetheless, we believe they remain sharp in Alexandrov spaces well beyond the zero-curvature setting. In particular, establishing information-theoretic lower bounds on estimator accuracy, comparable to the results of \cite{depersin2022optimal}, remains a challenging open problem. In Euclidean spaces, such results can be derived, for example, via the Gaussian shift theorem, but a natural analogue in the present setting is far from obvious.
Moreover, as noted in Section \ref{sec:frechet.vs.mean}, the local variance term implied by our main theorem is suboptimal in the $\ell_p$ case. This likely traces back to \eqref{eq:uniformlyconvex}, where the ``Hessian'' of $\|\cdot\|^2$  is estimated coarsely through uniform convexity.
Determining the correct bounds in this case, and clarifying how the difficulty (i.e., sample complexity) of estimating the usual mean compares to that of the Fréchet mean---even in $\ell_p$ spaces---, remain compelling open questions.


The computational aspects remain open at the level of generality studied in this paper. Our results, as well as those in the related literature, focus primarily on information-theoretically optimal guarantees. Even in the Euclidean setting, the picture is incomplete. Hopkins and Li \cite{hopkinsLee} provided evidence that no estimator can be simultaneously sub-Gaussian, optimally robust to adversarial contamination, and efficiently computable (in time $\tilde{O}(n^2)$).  There has been important progress on designing computationally efficient methods for high-dimensional robust statistics; see, for example, \cite{Yeshwanth, DeLe22, diakonikolas2025, DiKaKaLiMoSt19, diakonikolas2022, dong2019quantum, Hop20, Hopkins20}.

Many other questions on estimating Fréchet means remain to be explored. A particular concrete problem is characterizing the metric spaces in which the condition $\EXP d^2(X,x)< \infty$ for some $x\in \X$ implies the existence of a Fréchet mean estimator whose risk decreases at a rate of $O(n^{-1/2})$.

\section{Proof of Remark \ref{rem:MOM.general}}
\label{proof.MOM.general}

Set $k:=\lceil 8\log(1/\delta) \rceil$ and put $\ell:=\lfloor n/k\rfloor$.
Divide the data into $k$ disjoint blocks containing $\ell$ samples each, and let $Y_1,\ldots,Y_k$ denote the empirical Fr\'echet means in each block.

Set $\wh\mu_n$ to be the center of the smallest ball in $\X$ containing strictly more than $k/2$ of the points $Y_1,\ldots,Y_k$ and denote by $\wh\rho_n$ the radius of this ball.
Further, let $\rho_n$ denote the radius of the ball centered at $\mu$ that contains strictly more than $k/2$ of the $Y_j$'s. 
Clearly, $\wh\rho_n \le \rho_n$, and since the two balls intersect, we have $d(\wh\mu_n,\mu) \le \wh\rho_n+\rho_n \le 2\rho_n$.

To bound $\rho_n$, observe that by Markov's inequality and the definition of $R(\ell)$, for each $j\le k$,
\[
    \PROB\left( d(Y_j,\mu) \ge 4R(\ell) \right) \le \frac{1}{4}~.
\]
Therefore
\[
  \PROB( \rho_n \ge 4R(\ell) ) \le \PROB\left(\rm{Bin}(k,1/4) \ge \frac{k}{2} \right)
  \le e^{-k/8}~,
\]
Finally, since $k\geq 8\log(2)$ it follows that $k\leq \frac{7}{6} 8\log(1/\delta)$ and hence $\ell \geq \frac{1}{20} \frac{n}{\log(1/\delta)}$,  which completes the proof.
\qed

\vspace{1em}
\noindent
{\bf Acknowledgments:}
Daniel Bartl is grateful for financial support through the Austrian Science Fund (DOI: 10.55776/P34743 and 10.55776/ESP31), the Austrian National Bank (Jubil\"aumsfond, project 18983), and the NUS-PYP grant (`Robust Statistical Learning from Complex Data').
G\'abor Lugosi acknowledges the support of Ayudas Fundación BBVA a
Proyectos de Investigación Científica 2021 and
the Spanish Ministry of Economy and Competitiveness grant PID2022-138268NB-I00, financed by MCIN/AEI/10.13039/501100011033,
FSE+MTM2015-67304-P, and FEDER, EU.
Roberto I. Oliveira gratefully acknowledges support from a {\em Bolsa de Produtividade em Pesquisa} from CNPq, Brazil (\# 305765/2023-0); an Excellent Science - Marie Skłodowska-Curie Actions grant from the European Commission (DOI: 10.3030/101007705); and the following grants from Faperj, Rio de Janeiro, Brazil: {\em Cientista do Nosso Estado} (\# E-26/200.485/2023) and {\em Edital Inteligência Artificial} (\# E-26/290.024/2021).

\bibliographystyle{plain}
\bibliography{frechet}

\appendix 

\section{Appendix: Existence of Fr\'{e}chet means in favorable spaces}\label{proof.existence.barycenter}


In this section, we establish sufficient conditions for both the existence of Fréchet means, as well as the measurability of empirical Fréchet means. 
The condition for $(\X,d)$ that we assume throughout the paper is that it is {\em favorable} in the sense of the following definition.

\begin{definition}[Favorable space] 
\label{def:favorable} 
A metric space $(\X,d)$ is {\em favorable} if it is Polish and for each closed ball $B[o,R]=\{x\in\X : d(o,x)\leq R\}$ (with $o\in\X$ and $R>0$) there exists a topology $\cT_{o,R}$ over $\mathcal{X}$ such that 
\begin{itemize}
\item $(B[o,R],\cT_{o,R})$ is compact; and 
\item for every $x\in\X$, the map  $d(x,\cdot)\colon\X\to\R$  is lower semicontinuous with respect to $\cT_{o,R}$; that is, for all $u\in\R$ and all $x\in \mathcal{X}$, $\{y\in B[o,R]\,:\,d(x,y)>u\}\in\cT_{o,R}$.
\end{itemize}
 \end{definition} 

Favorable spaces include many interesting examples.
For instance, if  $(\X,d)$ is proper (i.e., every closed ball is compact), it is automatically favorable with $\cT_{o,R}$ equal to the usual topology. 
More interestingly, even for spaces that are not proper, there is often a natural ``weak topology''~that works for our purposes. 

\begin{example}
Consider a proper metric space $(E,\rho)$. 
Given $p\in [1,+\infty)$, let $\X =\mathcal{P}_p(E,\rho)$ be the space of probability measures over $(E,\Borel(E))$ with finite $p$-th moment and set $\rho:=\mathcal{W}_p$ to be the Wasserstein $L^p$-distance; see Section \ref{sec:example.wasserstein} for background. The space $(\mathcal{P}_p(E,\rho),\mathcal{W}_p)$ is Polish. On each ball $B[o,R]$ of this space, we use the topology $\cT_{o,R}$ induced by weak convergence of probability measures. Then each closed ball in $\mathcal{P}_p(E,\rho)$ is compact (because $(E,\rho)$ is proper). Moreover, the distance function $\mathcal{W}_p(\cdot,\mu)$ is lower semicontinuous under weak convergence, for any $\mu\in\mathcal{P}_p(E,\rho)$. It follows that $(\mathcal{P}_p(E,\rho),\mathcal{W}_p)$  is favorable.
\end{example}

\begin{example}Let $(\X,\|\cdot\|)$ be a separable reflexive Banach space and let $d$ be the metric induced by the norm $\|\cdot\|$. For each closed ball $B[o,R]$, we can take $\cT_{o,R}$ to be the standard weak topology over the ball. This makes $B[o,R]$ compact by the Banach-Alaoglu theorem. Moreover, for each $x\in\X$, the function $\|\cdot-x\|$ is lower semicontinuous for the weak topology. It follows that $(\X,d)$ is favorable.
\end{example}

The next result shows that favorable spaces are appropriate for our problem, in that Fr\'{e}chet means exist under minimum conditions.  

\begin{lemma}
\label{lem:existence.barycenter}
	Let $(\X,d)$ be favorable and assume that $X$ is a random element of $(\X,d)$ with $\E\, d^2(X,x_0)<\infty$ for some $x_0\in\X$.
	Then $X$ has a Fréchet mean, i.e., there exists $\mu\in\X$ that satisfies
	\[ \E d^2(X,\mu) = \inf_{b\in\X} \E \,d^2(X,b) . \]
\end{lemma}

\begin{proof}
The proof consists of two steps. The first one is to show that the infimum of $\E\,d^2(X,\cdot)$ over $\X$ coincides with the infimum over some closed ball $B[x_0,R]$. 
We then show that the restriction of  $\E\,d^2(X,\cdot)$ to this ball is lower semicontinuous in the topology $\cT_{x_0,R}$. Since the ball is compact in this topology, that will imply that $\E\,d^2(X,\cdot)$ achieves its infimum over the ball (and thus over the entirety of $\X$). 

\noindent{\em Step 1:} Let $R>0$. For any $x\in\X$ with $d(x,x_0)>R$,
\[d^2(X,x)\geq |d(x,x_0) - d(X,x_0)|^2\geq (R^2 - 2R\,d(X,x_0))\,{\bf 1}_{\{d(X,x_0)<R\}}.\]
Taking expectations, we obtain
\[\forall x\in X\backslash B[x_0,R]\,:\,\E\,d^2(X,x)\geq R^2\PROB(d(X,x_0)<R) -2R\,\E[d^2(X,x_0)\,{\bf 1}_{\{d(X,x_0)<R\}}].\]
When $R\to +\infty$, $\PROB(d(X,x_0)<R)\to 1$ and $\E[d^2(X,x_0)\,{\bf 1}_{\{d(X,x_0)<R\}}]\to \E[d^2(X,x_0)]<+\infty$. We obtain:
\[\lim_{R\to +\infty}\inf_{ x\in \X\backslash B[x_0,R]}\E\,d^2(X,x)=+\infty.\]
Therefore, for large enough $R$, 
\[\inf_{x\in \X}\E\,d^2(X,x) = \inf_{x\in B[x_0,R]}\E\,d^2(X,x).\]

\noindent{\em Step 2:} Fix $x_0$ and $R>0$ as above. We need to show that the restriction of the function $\psi(\cdot):=\E\,d^2(X,\cdot)$ to the ball $B[x_0,R]$ is lower semicontinuous in the $\cT_{x_0,R}$ topology. That is, we wish to show that, for any $u\geq 0$, the set 
\[E_{B[x_0,R]}(u):=\{x\in B[x_0,R]\,:\,\psi(x)>u\}\]
is open for the $\cT_{x_0,R}$ topology. That is equivalent to showing that each $x\in E_{B[x_0,R]}(u)$ is contained in an open neighborhood $O_{x,u}\in \cT_{x_0,R}$ with $O_{x,u}\subset E_{B[x_0,R]}(u)$. From now on, we fix $u\geq 0$, $x\in E_{B[x_0,R]}(u)$, and look for the corresponding $O_{x,u}\in\cT_{x_0,R}$. 

For this, we first need some preliminaries. Since $(\X,d)$ is Polish, one can find an increasing sequence of  sets $K_1\subset K_2\subset K_3\subset \cdots$ that are compact in the usual metric topology and satisfy $\PROB(X\in \cup_{n}K_n)=1$. Therefore, by Monotone Convergence,
\[\forall y\in\X\,:\,\psi_n(y)\nearrow \psi(y),\mbox{ where } \psi_n(y):=\E[d^2(X,y){\bf 1}_{\{X\in K_n\}}].\]
Recall that $x\in E_{B[x_0,R]}(u)$ has been fixed, so $\psi(x)>u$. By the above, we may take $n$ large enough and $\xi\in(0,1)$ small enough so that $\psi_n(x)>u+\xi$. 

Given $z\in K_n$, the fact that $(\X,d)$ is favorable implies that $d(z,\cdot)$ is lower-semicontinuous in the $\cT_{x_0,R}$ topology. It follows that, for any $r_z>0$, there exists an open neighborhood $A_z\in \cT_{x_0,R}$ with $x\in A_z$ such that:
\[\forall y\in A_z\,:\,d(z,y)\geq d(z,x) - r_z.
\]
For reasons to become clear soon, we choose $r_z>0$ so that  
\[6\,(d(z,x)+r_z)\,r_z\leq \xi.\]
Now let $B_z:=B(z,r_z)$ denote the open ball of radius $r_z$ around $z$ (with respect to the metric $d$). For $y\in A_z$ and $w\in B_z$, the triangle inequality guarantees $d(w,x)\leq d(z,x)+r_z$ and 
\[\forall y\in A_z\,:\,d(w,y)\geq d(w,x) - 3r_z.\]
Using that
\begin{align}
    \label{eq:simple.inequality.later.also}
    \forall \alpha,\beta,\gamma\geq 0\,:\, \alpha\geq \beta-\gamma\Rightarrow \alpha^2 \geq \beta^2 - 2\beta\gamma,
\end{align}
we obtain
\[\forall w\in B_z\,\forall y\in A_z\,:\,d^2(w,y)\geq d^2(w,x) - 6d(w,x)\,r_z.\]
By the choice of $r_z$ above, and recalling $d(w,x)\leq d(z,x)+r_z$ for $w\in B_z$, we obtain:
\begin{equation}\label{eq:lowerboundcompact}\forall z\in K_n\,\forall y\in A_z\,\forall w\in B_z\,:\, d^2(w,y)\geq d^2(w,x)-\xi.\end{equation}
The open balls $\{B_z\,:\,z\in K_n\}$ form an open covering of $K_n$ in the metric topology, and $K_n$ is compact with respect to that same topology. It follows that $K_n\subset \cup_{i=1}^\ell B_{z_i}$ for some choice of $\ell\in\N$ and $z_1,\ldots,z_\ell\in K_n$. 

To finish the proof, we show that 
\[O_{x,u}:=\bigcap_{i=1}^\ell A_{z_i}\]
is the open neighborhood of $x$ that we are looking for. We first notice that this set is open according to $\cT_{x_0,R}$ (i.e., $O_{x,u}\in \cT_{x_0,R}$), because it is the intersection of finitely many elements of the topology $\cT_{x_0,R}$. We also have $x\in O_{x,u}$ because $x\in A_z$ for all $z$.

Finally, \eqref{eq:lowerboundcompact} implies
\[\forall y\in O_{x,u}\,\forall w\in K_n\,:\, d^2(w,y)\geq d^2(w,x)-\xi,\]
which gives (after taking expectations)
\[\forall y\in O_{x,u}\,:\,\psi_n(y)\geq \psi_n(x)-\xi >u.\]
Since $\psi_n\leq \psi$, this implies that $\psi(y)>u$ for all $y\in O_{x,u}$, meaning that $O_{x,u}\subset  E_{B[x_0,R]}(u)$, as desired.
\end{proof}

\section{Appendix: Measurability issues}
The goal of this section is to address two measurability issues that appear in the main text of the paper:

\begin{itemize}

\item Given a Banach space $\X$ and a convex function $\Psi:\X\to\R$, is there a measurable function $g:\X\to\X^*$ such that, for each $x\in\X$, $g(x)$ is a subgradient of $\Psi$ at $x$?
\item Is there a measurable way to define the trimmed-mean estimator of the Fr\'{e}chet mean?
\end{itemize}

We will see that the answer to both questions is ``yes'', as a consequence of a more general framework that is introduced in the next subsection. 

\subsection{Measurable selection of minimizers: a general result}
The next result is a simple consequence of the Kuratowski--Ryll-Nardzewski measurable selection theorem \cite{kuratowski1965selector} (see also \cite[section 18.3]{aliprantis2006infinite} for a short proof) which seems generally useful for many problems related to statistical estimation.

\begin{theorem}\label{thm:measurableselection} Let $(\Omega,\mathcal{A})$ be a measurable space; $(\M,\rho)$ be a Polish space with Borel $\sigma$-field $\mathcal{B}_\M$; and $F:\Omega\times \M\to \R$ be a function such that $F(\cdot,m)$ is $\mathcal{A}$-measurable for all $m\in\M$. 

For $\omega\in\Omega$ and nonempty $B\subset \M$, write
\[\mathop{\rm argmin}_{m\in B}F(\omega,m) = \left\{m\in B\,:\,F(\omega,m) = \inf_{m'\in B}F(\omega,m')\right\}. \]
Suppose that the following properties hold for each fixed $\omega\in\Omega$:
\begin{enumerate}
\item the function $F(\omega,\cdot)$ is continuous;
\item $F(\omega,\cdot)$ achieves its infimum over $\M$; that is, 
\[\mathop{\rm argmin}_{m\in\M}F(\omega,m)\neq \emptyset;\]
\item for each closed ball $B\subset \M$, $F(\omega,\cdot)$ achieves its infimum over $B$; that is 
\[\mathop{\rm argmin}_{m\in B}F(\omega,\cdot)\neq \emptyset.\]
\end{enumerate}
Then there exists a $\mathcal{A}/\mathcal{B}_\M$-measurable function $m_\star:\Omega\to \M$ such that, for all $\omega\in\Omega$, \[m_\star(\omega)\in \mathop{\rm argmin}_{m\in\M}F(\omega,m).\] 
\end{theorem}

\begin{proof} For $\omega\in\Omega$, we write $S(\omega):= {\rm argmin}_{m\in\M}F(\omega,m)$ for convenience.  $S:\Omega\to 2^{\M}$ defines a set-valued mapping. Items $1$ and $2$ in our assumptions imply that, for each $\omega\in\Omega$, $S(\omega)$ is a closed, non-empty subset of $\M$.

In what follows, we will argue that $S$  is a {\em weakly measurable set-valued mapping}. That is, we will show that if we define 
\[S^{-1}(V):=\{\omega\in\Omega\,:\, S(\omega)\cap V\neq \emptyset\}\,, \quad (V\subset \M), \]
then $S^{-1}(V)\in\mathcal{A}$ whenever $V$ is an {\em open} subset of $\M$. Once we have this, the existence of $m_\star$ follows from the Kuratowski-Ryll-Nardzewski measurable selection theorem and the observations in the preceding paragraph. 

To prove that $S$ is weakly measurable, fix an open subset $V\subset \M$. Because $\M$ is Polish, we can write $V= \cup_{n\in\N}B_n$ where each $B_n$ is a closed ball. One can easily check that
\[S^{-1}(V) = \bigcup_{n\in\N}S^{-1}(B_n).\]
Since $\mathcal{A}$ is closed under countable unions, we will be done if we can show that $S^{-1}(B)\in\mathcal{A}$ whenever $B\subset \M$ is a closed ball. 

To this end, fix  such $B$. Notice that, for any $\omega\in \Omega$:
\[\omega\in S^{-1}(B) \Leftrightarrow \exists x\in B\,:\, F(\omega,x) = \inf_{m\in \M}F(\omega,m).\]
Crucially, an $x$ as in the RHS of the above expression must also achieve the infimum of $F(\omega,\cdot)$ over the ball $B$, which must then equal the infimum over the whole of $\M$. It follows that  
\[\omega\in S^{-1}(B)\Leftrightarrow ``\inf_{m\in B}F(\omega,m) =  \inf_{m\in \M}F(\omega,m)\mbox{ and }\exists x\in B\,:\, F(\omega,x) = \inf_{m\in B}F(\omega,m)".\]
From item $3$ in our assumptions, we know that $F(\omega,\cdot)$ it guaranteed to achieve its infimum over $B$. Therefore,
\[\omega\in S^{-1}(B) \Leftrightarrow\inf_{m\in B} F(\omega,m) = \inf_{m\in \M}F(\omega,m).\]
In set notation, this becomes
\[S^{-1}(B)=\left\{\omega\in\Omega\,:\,\inf_{m\in B} F(\omega,m) = \inf_{m\in \M}F(\omega,m)\right\}.\]
To complete the proof that $S^{-1}(B)\in\mathcal{A}$, it suffices to show that $\inf_{m\in B} F(\omega,m)$ and $ \inf_{m\in \M}F(\omega,m)$ are $\mathcal{A}$-measurable functions of $\omega\in \Omega$. 

To prove this last point, let $D\subset \M$ and $Q_B\subset B$ be countable dense subsets (which exist because $\M$ is Polish). Since $F(\omega,\cdot)$ is continuous in the second coordinate, 
\[\forall \omega\in\Omega\,:\,\inf_{m\in B} F(\omega,m) = \inf_{m\in Q_B} F(\omega,m)\mbox{ and } \inf_{m\in \M}F(\omega,m) = \inf_{m\in D}F(\omega,m).\]
Therefore, $\inf_{m\in B} F(\cdot,m)$ and $ \inf_{m\in \M}F(\cdot,m)$ are infima of {\em countably many} $\mathcal{A}$-measurable functions. As a consequence, both infima are $\mathcal{A}$-measurable, as we wanted to show.\end{proof}

\subsection{Measurability of subgradients}

In this subsection, we let $(\X,\|\cdot\|)$ denote a separable Banach space with separable dual $(\X^*,\|\cdot\|_*)$. Given a function $\Psi:U\subset\X\to \R$, recall that a vector $g_*\in\X^*$ is said to be a subgradient of $\Psi$ at $x\in\X$ if 
\[\forall y\in\X\,:\, \Psi(y) - \Psi(x)\geq \langle g_*,y-x\rangle,\]
where $\langle \cdot,\cdot\rangle:\X^*\times \X\to \R$ is the natural pairing of $\X^*$ and $\X$. 

\begin{theorem}
\label{thm:measurable.subgradient}
Assume $U\subset \X$ is open and convex and $\Psi:U\to\R$ is convex and continuous. Then there exists a $\mathcal{B}_U/\mathcal{B}_{\X^*}$-measurable function $g_\star:U\to\X^*$ such that, for all $x\in U$, $g_\star(x)$ is a subgradient of $\Psi$ at $x$. (Here $\mathcal{B}_U$ and $\mathcal{B}_{\X^*}$ are the natural Borel $\sigma$-fields given by the respective norm topologies over $U$ and $\X^*$, respectively.)\end{theorem}

The idea of the proof will be to apply Theorem \ref{thm:measurableselection} for a suitable $F:\X\times \X^*\to \R$ such that $g_*$ is a subgradient at $x$ if and only if $F(x,g_*)=\inf_{g_*'\in\X^*}F(x,g_*)$. 

To construct this function, we will need the following proposition. 

\begin{proposition}\label{prop:localsubgrad} $g_\ast\in \X^*$ is a subgradient of $\Psi$ at $x\in U$ if and only if there exists $r>0$ such that $B_{(\X,\|\cdot\|)}[x,r]\subset U$ and for all $h\in B_{(\X,\|\cdot\|)}[0,1]$
\[\Psi(x+rh) - \Psi(x)\geq r\, \langle g_*,h\rangle.\]
\end{proposition}
\begin{proof}
The necessity of this condition is clear. To show sufficiency, assume $x$ and $g_*$ are fixed and that the above condition holds. Given $y\in U$, we wish to show that $\Psi(y) - \Psi(x)\geq \langle g_*,y-x\rangle$. This is direct if $y\in B_{(\X,\|\cdot\|)}[x,r]$, as in that case $y=x+rh$ with $\|h\|\leq 1$. If on the other hand $y\not\in B_{(\X,\|\cdot\|)}[x,r]$, one still has
\[y := x + \|y-x\|\,h\,\mbox{ with } h:= \frac{y-x}{\|y-x\|}\in B_{(\X,\|\cdot\|)}[0,1].\]
Define $y_r:= x + rh$. Since $0<r<\|y-x\|$, $y_r$ lies in the line segment from $x$ to $y$. Therefore, convexity gives: 
\[\frac{\Psi(y) - \Psi(x)}{\|y-x\|}\geq \frac{\Psi(x+rh)-\Psi(x)}{r}\geq \frac{\langle g_*,h\rangle}{r} = \left\langle g_*,\frac{y-x}{\|y-x\|}\right\rangle,\]
where in the last two steps we used our assumption and the definition of $h$.\end{proof}

\begin{proof}[Proof of the Theorem \ref{thm:measurable.subgradient}]
First observe that there exists a bounded continuous function $r:U\to (0,1]$ such that $B_{(\X,\|\cdot\|)}[x,r(x)]\subset U$ for all $x\in U$. Indeed, it suffices to take $r(x) \equiv 1$ if $U=\X$, and
\[r(x) :=\max\left\{1,\frac{1}{2}\inf_{y\in \X-U}\|x-y\|\right\}\\,\quad (x\in U)\]
if $U\neq \X$. 

We now define $G:U\times \X^*\times B_{(\X,\|\cdot\|)}[0,1]\to \R$ and $F:U\times \X^*\to\R$ as follows: given $(x,g_*,h)\in U\times \X^*\times B_{(\X,\|\cdot\|)}[0,1]$,
\[G(x,g_*,h):=\min\{1,r(x)\,\langle g_*,h\rangle - \Psi(x+r(x)h) + \Psi(x)\}\]
and 
\[F(x,g_*):= \sup_{h\in B_{(\X,\|\cdot\|)}[0,1]}G(x,g_*,h).\]
The function $F$ is nonnegative for all $(x,g_*)\in U\times \X^*$ because $F(x,g_*)\geq G(x,g_*,0)=0$.
Moreover, Proposition \ref{prop:localsubgrad}, combined with the choice of $r(\cdot)>0$, implies that $g_*\in \X^*$ is a subgradient of $\Psi$ at $x\in U$ if and only if $F(x,g_*)=0$. Since $F$ is nonnegative, we deduce that for all $(x,g_*)\in U\times \X^*$, 
\[
\mbox{$g_*\in\X^*$ is a subgradient of $\Psi$ at $x\in\X$}\Leftrightarrow g_*\in \mathop{\rm argmin}_{g'_*\in\X^*}F(x,g_*'),\]
as indicated in the beginning of the proof. 

To obtain the measurable choice of subgradient, we will apply Theorem \ref{thm:measurableselection} to this function $F$, with the choices:
\begin{itemize}
\item $\Omega=U$;
\item $\mathcal{A}=\mathcal{B}_U$, the Borel $\sigma$-field on $U$ (with the usual norm topology);
\item $\M=\X^*$ with the metric $\rho$ given by the dual norm $\|\cdot\|_*$ (this is a Polish space because $(\X^*,\|\cdot\|_*)$ is assumed separable).\end{itemize}

For the remainder of the proof, we check that the remaining assumptions of Theorem \ref{thm:measurableselection} are satisfied. First, we argue that, for a fixed $g_*\in\X^*$, $F(\cdot,g_*)$ is $\mathcal{B}_{U}$-measurable. This follows from the fact that $\Psi$ and $r(\cdot)$ are continuous functions (thus measurable), so that $G(x,g_*,h)$ depends continuously on $(x,h)$. For this reason, the supremum over all $h\in B_{(\X,\|\cdot\|)}[0,1]$ in the definition of $F(x,g_*)$ can be replaced by a supremum over a countable dense subset of the unit ball (since $\X$ is separable; again we use continuity of $\Psi$). As a consequence, $F(\cdot,g_*)$ is the pointwise supremum of countably many $\mathcal{B}_{\X}$-measurable functions. It follows that $F(\cdot,g_*)$ is $\mathcal{B}_\X$-measurable, as desired.

We now fix $x\in U$ and check conditions 1-3 in Theorem \ref{thm:measurableselection}. Condition 1 (continuity) follows because $F(x,\cdot)$ is a supremum of uniformly bounded, $r(x)$-Lipschitz functions. Condition 2 (the infimum is achieved) follows from the existence of subgradients for a convex function over all points of an open and convex set. 

It remains to check condition 3, that $F(x,\cdot)$ achieves its infimum over any closed ball of the dual space $\X^*$. This holds because  $F(x,g_*)$ is the pointwise supremum of $G(x,\cdot,h) $, which are {\em weak-$*$ continuous functions} and uniformly bounded from above. Therefore, $F(x,\cdot)$ is weak-$*$ lower-semicontinuous, and achieves its infimum over any weak-$*$ compact set. The Banach-Alaoglu Theorem guarantees that closed balls of $\X^*$ are weak-$*$ compact, so condition 3 follows.\end{proof}

\subsection{Measurable empirical trimmed Fr\'{e}chet means}\label{sub:measurabletrimmed}

We now consider the question of whether the estimator we defined in the main text is a measurable function of the data. This will be shown to be the case whenever the metric space is {\em favorable} in the sense of Definition \ref{def:favorable} in the main text.

Fix $n\in\N$, $t\in\N\cup\{0\}$ with $2t<n$. Define 
$T_{n,t}:\X^{n+2}\to \R$ as follows: for $(x_1,\ldots,x_n)\in \X^n$ and $(a,b)\in\X^2$,
\begin{align*}
T_{n,t}(x_1,\ldots,x_n,a,b) & := \frac{1}{n-2t}\sum_{i=t+1}^{n-2t}\,\left((d^2(x_{j},b) - d^2(x_{j},a))_{j=1}^n\right)_{(i)}.
\end{align*}

A {\em $t$-trimmed empirical Fr\'{e}chet-mean} of $(x_1,\ldots,x_n)\in\X^n$ is any element of the set of minimizers 
\[\mu\in \mathop{\rm argmin}_{b\in\X}\,\left(\sup_{a\in \X}T_{n,t}(x_1,\ldots,x_n,a,b)\right).\]  
For $t=0$, this coincides with the usual definition of an empirical Fr\'{e}chet mean of the sample $(x_1,\ldots,x_n)$. 

We say that $(\X,d)$ {\em has measurable trimmed empirical Fr\'{e}chet means} if, for any $t$ and $n$ as above, there exists $\wh\mu_{n,t}:\X^n\to\X$ such that, for any $(x_1,\ldots,x_n)\in\X^n$, $\wh\mu_{n,t}(x_1,\ldots,x_n)$ is a $t$-trimmed Fr\'{e}chet mean of $(x_1,\ldots,x_n)$. We prove the following theorem.

\begin{theorem}[Favorable spaces have measurable trimmed empirical Fr\'{e}chet means] \label{thm:favorablemeasurable} If the Polish space $(\X,d)$ is favorable (as in Definition \ref{def:favorable} above), then it has measurable trimmed empirical Fr\'{e}chet means.\end{theorem}

The proof of Theorem \ref{thm:favorablemeasurable} will require the machinery of Theorem \ref{thm:measurableselection}. Fix $n\in\N$ and $0\leq t<n/2$. We will apply Theorem \ref{thm:measurableselection} with $\Omega:=\X^{n}$, $\mathcal{A}$ the Borel (product) $\sigma$-field coming from the metric $d$; and $(\M,\rho):=(\X,d)$. We define: 
\[F_{n,t}:\X^{n}\times \X\to \R\]
via
\[F_{n,t}((x_1,\ldots,x_n),b):= \sup_{a\in \X}T_{n,t}(x_1,\ldots,x_n,a,b)\, ,\ \quad ((x_1,\ldots,x_n)\in \X^n,\, b\in \X).\]

The bulk of the proof consists of three steps, which we state as separate propositions. 

\begin{proposition}\label{prop:measurablefavorable1}For fixed $b\in \X$, $F_{n,t}(\cdot,b)$ is Borel-measurable.\end{proposition}
\begin{proof}Since $T_{n,t}(x_1,\ldots,x_n,a,b)$ is continuous in $b$ and $(\X,d)$ is Polish, the supremum in the definition of $F_{n,t}$ can be restricted to a countable dense subset of values $a$. As a result, $F_{n,t}(\cdot,b)$ is a supremum of countably many Borel-measurable functions.
\end{proof}

\begin{proposition}\label{prop:continuousfavorable}
For every  $x_1,\ldots,x_n\in\X$, the function $F_{n,t}((x_1,\ldots,x_n),\cdot)$ is continuous.
Moreover, given $o\in\X$ and $R>0$ such that $x_1,\ldots,x_n\in B[o,R]$,
\[F_{n,t}((x_1,\ldots,x_n),o)< \inf_{b\in \X\setminus B[o,3R]}F_{n,t}((x_1,\ldots,x_n),b).\]
\end{proposition}
\begin{proof}
For fixed $x_1,\ldots,x_n$, the family of functions $(T_{n,t}((x_1,\ldots,x_n),\cdot,a))_{a\in \X}$ is uniformly continuous and uniformly bounded over closed balls. Therefore, $F_{n,t}$ is continuous. Now note that
\[F_{n,t}((x_1,\ldots,x_n),o)\leq R^2\]
whereas for $b\in \X\setminus B[o,3R]$,
\[F_{n,t}((x_1,\ldots,x_n),b)\geq T_{n,t}(x_1,\ldots,x_n,o,b)\geq 3R^2.\]\end{proof}

\begin{proposition}\label{prop:minoverclosedballsfavorable}For any closed ball $B[o,R]\subset\X$ and any $x_1,\ldots,x_n\in\X$, 
\[\inf_{b\in B}F_{n,t}((x_1,\ldots,x_n),b)\mbox{ is achieved}.\]
\end{proposition}

\begin{proof}Fix $(x_1,\ldots,x_n)\in\X^n$, $o\in\X$ and $R>0$. Since $(\X,d)$ is favorable, $B[o,R]$ is compact in the topology $\cT_{o,R}$. Therefore, to prove that $F_{n,t}((x_1,\ldots,x_n),\cdot)$ achieves its infimum over $B[o,R]$, it suffices to show that this  is lower semicontinuous over $B[o,R]$ in the $\cT_{o,R}$ topology. In fact, since 
\[F_{n,t}((x_1,\ldots,x_n),b) = \sup_{a\in \X}T_{n,t}(x_1,\ldots,x_n,a,b),\]
and the supremum of lower semicontinuous functions is also lower semicontinuous, it suffices to show that the map
\[b\in B[o,R]\mapsto T_{n,t}(x_1,\ldots,x_n,a,b)\]
is lower semicontinuous with respect to $\cT_{o,R}$, for each choice of $x_1,\ldots,x_n,a\in\X.$ 
That is, our goal is to show that, for any $u\in\R$, the set
\[O_u:=\{b\in B[o,R]\,:\,T_{n,t}(x_1,\ldots,x_n,a,b)>u\}\]
belongs to the topology  $\cT_{o,R}$. 

To do this, take an arbitrary $b\in O_{u}$ and fix
\[\delta:=  T_{n,t}(x_1,\ldots,x_n,a,b)-u>0.\]
 Because the distance functions are lower semicontinuous, one can find, for each $1\leq i \leq n $, an open neighborhood $U_i\in \cT_{o,R}$ with $b\in U_i$ such that
\[\forall i\leq n\,\forall b'\in U_i\,:\, d(b',x_i)> d(b,x_i)-\frac{\delta}{1+4d(b,x_i)}.\]
Then $U:=\cap_{i\leq n}U_i\in\cT_{o,R}$ is also open,  $b\in U$, and by \eqref{eq:simple.inequality.later.also},
\[\forall b'\in U\,\forall i\leq n\,:\, d^2(b',x_i)> d^2(b,x_i) - \frac{\delta}{2}.\]
As a consequence,
\[\forall b'\in U\,\forall i\leq n\,:\, \left( (d^2(b',x_{j}) - d^2(a,x_j) \right)_{j=1}^n)_{(i)}>  \left( (d^2(b,x_{j}) - d^2(a,x_j) )_{j=1}^n \right)_{(i)}- \frac{\delta}{2}.\]
We deduce:
\[\forall b'\in U\,:\, T_{n,t}(x_1,\ldots,x_n,a,b')\geq T_{n,t}(x_1,\ldots,x_n,a,b) - \frac{\delta}{2}>u.\]
To conclude, we have shown that for an arbitrary $b\in O_u$ there exists an open neighborhood $U\in\cT_{o,R}$ with $b\in U\subset O_u$. Since $b\in O_{u}$ is arbitrary, $O_{u}\in \cT_{o,R}$, as desired.
\end{proof}

\begin{proof}[Proof of Theorem \ref{thm:favorablemeasurable}] We use the notation introduced above: that is, we take $n\in\N$  and $0\leq t<n/2$, and set $\Omega:=\X^{n}$, $\mathcal{A}=$ the Borel product $\sigma$-field, $(\M,\rho)=(\X,d)$, and $F=F_{n,t}$. As we will see below, all assumptions of Theorem \ref{thm:favorablemeasurable} hold for this example. The (measurable) minimizer of $F((x_1,\ldots,x_n),\cdot)$, which is guaranteed to exist by that theorem, is the measurable trimmed mean $\widehat{\mu}_{n,t}(x_1,\ldots,x_n)$ we are looking for.  

We check that the assumptions of Theorem \ref{thm:measurableselection} are satisfied. To start, $F_{n,t}$ is measurable in the first variable~is shown in Proposition \ref{prop:measurablefavorable1}. 

When $x_1,\ldots,x_n\in \X$ are fixed, continuity in $a\in\X$ -- condition 1 in Theorem \ref{thm:measurableselection} -- is shown in Proposition \ref{prop:continuousfavorable}. Proposition \ref{prop:minoverclosedballsfavorable} shows that $F_{n,t}((x_1,\ldots,x_n),\cdot)$ achieves its infimum over closed balls, which is condition 3 in Theorem \ref{thm:measurableselection}. 

It remains to check condition 2, that is, that $F_{n,t}((x_1,\ldots,x_n),\cdot)$ has a global minimizer. To do this, notice that the second assertion of Proposition \ref{prop:continuousfavorable} implies that, with a suitable choice of $o\in\X$ and $R>0$, a minimizer of $F_{n,t}((x_1,\ldots,x_n),\cdot)$ over the closed ball $B[o,3R]$ must be a global minimizer as well. Since minimizers over closed balls exist,  condition 2 in Theorem \ref{thm:measurableselection} obtains, and the proof is finished.\end{proof}

\end{document}